\documentclass[12pt]{amsart}
\usepackage{amsmath}
\usepackage{amscd}
\usepackage{amssymb}
\usepackage{amsfonts}
\usepackage{graphicx}
\usepackage{wrapfig}
\usepackage{amsthm}
\usepackage{euscript}
\usepackage[utf8]{inputenc}
\usepackage[english]{babel}
\usepackage{pgf,tikz}
\usepackage[colorlinks, citecolor=blue]{hyperref}

\def\<{\langle}
\def\>{\rangle}

\newtheorem{thm}{Theorem}

\newtheorem{theorem}{Theorem}[section]
\newtheorem{lemma}[theorem]{Lemma}
\newtheorem{corollary}[theorem]{Corollary}
\newtheorem{proposition}[theorem]{Proposition}

\theoremstyle{definition}

\theoremstyle{remark}
\newtheorem{remark}[theorem]{Remark}

\numberwithin{equation}{section}

\usepackage[utf8]{inputenc}
\usepackage{tikz}
\usetikzlibrary{patterns}
\pgfdeclarepatternformonly{my dots}{\pgfqpoint{-1pt}{-1pt}}{\pgfqpoint{2.5pt}{2.5pt}}{\pgfqpoint{6pt}{6pt}}%
{
	\pgfpathcircle{\pgfqpoint{0pt}{0pt}}{.5pt}
	\pgfpathcircle{\pgfqpoint{7pt}{7pt}}{.5pt}
	\pgfusepath{fill}
}
\begin{document}

\title[Willmore-type inequalities]
{Willmore-type inequalities for closed hypersurfaces\\
in weighted manifolds}

\author{Guoqiang Wu}
\address{School of Science, Zhejiang Sci-Tech University, Hangzhou 310018, China}
\email{gqwu@zstu.edu.cn}

\author{Jia-Yong Wu}
\address{Department of Mathematics and Newtouch Center for Mathematics, Shanghai University, Shanghai 200444, China}
\email{wujiayong@shu.edu.cn}

\dedicatory{}
\date{\today}

\begin{abstract}
In this paper, we prove some Willmore-type inequalities for closed hypersurfaces in
weighted manifolds with nonnegative Bakry-\'Emery Ricci curvature. In particular,
we give a sharp Willmore type inequality in steady gradient Ricci solitons.
We also prove a sharp Willmore-like inequality in shrinking gradient Ricci solitons.
Moreover, we characterize the equality cases of Willmore-type inequalities. These
results can be regarded as weighted versions of Agostiniani-Fogagnolo-Mazzieri's
Willmore-type inequality. As applications, we derive some sharp isoperimetric type
inequalities in weighted manifolds under the existence assumption of a critical set
of weighted isoperimetric functional.
\end{abstract}
\maketitle

\section{Introduction}\label{Int1}

The classical Willmore inequality \cite{Wil} states that any a bounded domain
$\Omega$ in $3$-dimensional Euclidean space $\mathbb{R}^3$ with smooth
boundary $\partial\Omega$ satisfies
\[
\int_{\partial\Omega}
\left(\frac{H}{2}\right)^2d\sigma\ge4\pi,
\]
where $H$ is the mean curvature of $\partial\Omega$ and $d\sigma$ is the
Riemannian volume element of $\partial\Omega$ induced by the standard
Euclidean metric. Moreover, the equality occurs if and only if $\Omega$
is a $3$-dimensional round ball. The  classical Willmore inequality was extended
by Chen \cite{Chen,Chen2} to submanifolds of any co-dimension in $n$-dimensional
($n\ge3$) Euclidean space $\mathbb{R}^n$. In particular, one has that any
a bounded domain $\Omega$ in $\mathbb{R}^n$ with smooth boundary $\partial\Omega$
must have
\[
\int_{\partial\Omega} \left|\frac{H}{n-1}\right|^{n-1}
d\sigma\geq|\mathbb{S}^{n-1}|,
\]
where $|\mathbb{S}^{n-1}|$ is the area of $(n-1)$-dimensional Euclidean
unit sphere $\mathbb{S}^{n-1}$. Moreover, the equality holds if and only
if $\Omega$ is an $n$-dimensional round ball. It is worth pointing out that the
Willmore inequality can be reproved by Agostiniani and Mazzieri \cite{AM2}
via a monotonicity formula approach in the potential theory. Furthermore,
they gave a sharp quantitative version of the classical Willmore-type
inequality. See also an alternative proof of the Willmore inequality via
a geometric divergence inequality exploited by Cederbaum and Miehe \cite{CM}.
In \cite{AFM}, Agostiniani, Fogagnolo and Mazzieri generalized the Willmore
inequality to a bounded and open domain in Riemannian manifolds with
nonnegative Ricci curvature. Their result is stated as follows.

\begin{thm}[\cite{AFM}]\label{ACthm}
Let $(M^n,g)$ ($n\geq3$) be a complete noncompact Riemannian manifold
with nonnegative Ricci curvature and Euclidean volume growth. If
$\Omega\subset M^n$ is a bounded and open subset with smooth boundary
$\partial\Omega$, then
\begin{equation}\label{Willineq}
\int_{\partial\Omega}\left|\frac{H}{n-1}\right|^{n-1}d\sigma
\ge\mathrm{AVR}(g)|\mathbb{S}^{n-1}|,
\end{equation}
where $\mathrm{AVR}(g)$ is the asymptotic volume ratio of $(M^n,g)$.
Moreover, the equality holds if and only if $(M^n\backslash\Omega,g)$ is isometric to
\[
\left([r_0,\infty)\times\partial\Omega,dr^2+(\tfrac{r}{r_0})^2g_{\partial\Omega}\right),
\]
with $r_0=\left(\frac{|\partial\Omega|}{\mathrm{AVR}(g)|\mathbb{S}^{n-1}|}\right)^{\frac{1}{n-1}}$.
In particular, $\partial\Omega$ is a connected totally umbilic submanifold
with constant mean curvature.
\end{thm}

The proof idea of Theorem \ref{ACthm} in \cite{AFM} is similar to
\cite{AM2}, which is based on the monotonicity-rigidity properties
on certain level set flow of the electrostatic potentials associated
with $\Omega$. Later, Wang \cite{Wa} gave a short proof by using standard
comparison methods of the Riccati equation in Riemannian geometry.
Motivated by Wang's argument, Rudnik \cite{Ru} applied comparison methods
of Jacobi equations to obtain a Willmore-type inequality in manifolds with
asymptotically nonnegative curvature. In \cite{BF}, Borghini and Fogagnolo
proved a Willmore-like inequality in substatic manifolds by exploiting the
substatic Bishop–Gromov monotonicity theorem. Recently, Jin and Yin \cite{JY}
used Wang's method to extend Theorem \ref{ACthm} to the manifold of negative
Ricci curvature. Besides, many Willmore-type inequalities in various ambient
spaces are explored in \cite{AM1,Hu,Sc,V} and references therein.

In this paper, using Wang's argument \cite{Wa}, we shall prove Willmore-type
inequalities for closed hypersurfaces in weighted manifolds with nonnegative
Bakry-\'Emery Ricci curvature, extending Theorem \ref{ACthm} to weighted
manifolds. Our results imply a sharp Willmore-type inequality in steady
gradient Ricci solitons. We also prove a sharp Willmore-like inequality in
shrinking gradient Ricci solitons. Moreover, we characterize the equality
cases of Willmore-type inequalities. As applications, we give some sharp
isoperimetric type inequalities in weighted manifolds admitting a critical
set of weighted isoperimetric functional.

Recall that a weighted $n$-manifold, denoted by $(M^n,g,e^{-f}dv)$, is an
$n$-dimensional complete Riemannian manifold $(M^n,g)$ coupled with a weighted
measure $e^{-f}dv$ for some smooth weight function $f$ and the Riemannian volume
element $dv$ on $(M^n,g)$. Weighted manifolds are natural extensions of Riemannian
manifolds and are characterized by collapsed measured Gromov-Hausdorff limits
\cite{Lo}. They are also closely related to the Ricci flow, probability theory
and optimal transport; see \cite{Lo,LV} and references therein.

On $(M^n,g,e^{-f}dv_g)$, Bakry and \'Emery \cite{BE} introduced the
$m$-Bakry-\'Emery Ricci curvature
\[
\mathrm{Ric}_f^m:=\mathrm{Ric}+\mathrm{Hess}\,f-\frac{df\otimes df}{m-n}
\]
for some real number $m\ge n$, where $\mathrm{Ric}$ is the Ricci tensor of
$(M^n,g)$ and $\mathrm{Hess}$ is the Hessian with respect to $g$. When $m=n$,
function $f$ should be regard as a constant and $\mathrm{Ric}_f^m$ returns
to the ordinary Ricci curvature. When $m<\infty$, there exists a basic
viewpoint that many geometric results for manifolds with Ricci tensor
bounded below can be possibly extended to weighted manifolds with
$m$-Bakry-\'Emery Ricci curvature bounded below. This is because the
Bochner formula for $\mathrm{Ric}_f^m$ can be regarded as the classical
Bochner formula for $\mathrm{Ric}$  of an $m$-manifold; see for example
\cite{WW} for detailed explanations. In particular, a weighted manifold
satisfying
\[
\mathrm{Ric}_f^m=\lambda g
\]
for some $\lambda\in\mathbb{R}$, is called a quasi-Einstein $m$-manifold
(see \cite{CSW}), which is considered as the generalization of Einstein
manifold. When $n<m<\infty$, $(M^n\times F^{m-n}, g_M+e^{\frac{-2f}{m-n}}g_F)$,
with $(F^{m-n},g_F)$ an Einstein manifold, is a warped product Einstein
manifold. When $m=n+1$, $(M^n, g,e^{-f}dv)$ is the so-called static Einstein
manifold in general relativity. When $m\to\infty$, $\mathrm{Ric}_f^m$ becomes
the $\infty$-Bakry-\'Emery Ricci curvature
\[
\mathrm{Ric}_f:=\lim_{m\to\infty}\mathrm{Ric}_f^m=\mathrm{Ric}+\mathrm{Hess}\,f.
\]
Similar to classical comparison theorems, when $\mathrm{Ric}_f$ is bounded below,
one can also prove weighted comparison theorems (see \cite{WW,WWu2}), which highly
rely on $f$. This leads to many classical geometric and topological results
remaining true in weighted manifolds under certain assumptions on $f$. We refer
the readers to \cite{FLZ,FW,MuWa,WW,Wu,Wu1,Wu2,Wu3,WWu1,WWu2} and references therein for
nice works in this direction. In particular, if
\[
\mathrm{Ric}_f=\lambda\, g
\]
for some $\lambda\in\mathbb{R}$, then $(M^n, g, e^{-f}dv_g)$ is a gradient Ricci
soliton. A gradient Ricci soliton is called shrinking, steady, or expanding,
if $\lambda>0$, $\lambda=0$, or $\lambda<0$, respectively. The gradient Ricci
soliton plays an important role in the Ricci flow and Perelman's resolution
of the Poincar\'e conjecture; see \cite{Cao,Ham} and references therein for
nice surveys.

On $(M^n,g,e^{-f}dv)$, a natural generalization of Laplacian $\Delta$
is the $f$-Laplacian defined by
\[
\Delta_f:=\Delta-\langle\nabla f,\nabla\rangle,
\]
where $\nabla$ is the the gradient operator on $(M^n,g)$. The $f$-Laplacian
is self-adjoint with respect to the measure $e^{-f}dv$, and it is related
to $\mathrm{Ric}_f$ via the generalized Bochner formula
\[
\Delta_f|\nabla u|^2=2|\mathrm{Hess}\,u|^2+2\langle\nabla\Delta_f u, \nabla u\rangle+2\mathrm{Ric}_f(\nabla u, \nabla u)
\]
for any $u\in C^\infty(M^n)$. This formula will play an important role in our
paper. For a point $p\in M^n$, in this paper we let $r(x):=d(x,p)$ be a
distance function from $p$ to $x\in M^n$. In geodesic polar coordinates at $p$,
let $\nabla r=\partial r$ and $|\nabla r|=1$ almost everywhere. Let $B(p,r)$ be
the geodesic ball with center $p\in M^n$ and radius $r>0$, and its weighted volume
is given by $|B(p,r)|_f:=\int_{B(p,r)}e^{-f}dv$. Sometimes we denote $|B(p,r)|_f$
by $\mathrm{Vol}_f(B(p,r))$.

On $(M^n,g,e^{-f}dv_g)$, let $\Omega\subset M^n$ be a bounded open set with
smooth boundary $\partial\Omega$. We denote by $|\Omega|_f$ the weighted
volume of $\Omega$. When $f=0$, it returns to the Riemannian volume of
$\Omega$, denoted by $|\Omega|$. Let $H$ be the mean curvature of
$\partial\Omega$ and $\nu$ be the unit outer normal vector of $\partial\Omega$.
The weighted mean curvature of $\partial\Omega$ is defined by
\[
H_f:=H-g(\nabla f,\nu).
\]
When $\mathrm{Ric}_f\ge 0$ and $H_f$ of $\partial\Omega$ is nonnegative
everywhere, for each point $p\in \partial\Omega$, if
\begin{equation}\label{assump2}
\partial_rf\ge -a
\end{equation}
for some constant $a\ge0$ along all minimal geodesic segments from
$p$, inspired by \cite{JY}, we define the weighted relative volume
\[
\mathrm{RV}_f(\Omega):=\lim_{r\to\infty}\frac{\mathrm{Vol}_f\{x\in M|d(x,\Omega)<r\}}{m(r)},
\]
where $m(r):=n |\mathbb{B}^n|\int^r_0e^{at}t^{n-1}dt$ and $|\mathbb{B}^n|$
is the volume of Euclidean unit $n$-ball. By Proposition \ref{welldef}(a),
$\mathrm{RV}_f(\Omega)$ is well defined. For each point $p\in \partial\Omega$,
if
\begin{equation}\label{assump2b}
|f|\le k
\end{equation}
for some constant $k\ge0$ along all minimal geodesic segments from
$p$, we can define another weighted relative volume
\[
\overline{\mathrm{RV}}_f(\Omega):=\lim_{r\to\infty}\frac{\mathrm{Vol}_f\{x\in M|d(x,\Omega)<r\}}
{|\mathbb{B}^{n+4k}|r^{n+4k}}.
\]
By Proposition \ref{welldef}(b), $\overline{\mathrm{RV}}_f(\Omega)$
is also well defined. In general, $\mathrm{RV}_f(\Omega)$ and
$\overline{\mathrm{RV}}_f(\Omega)$ depend on the set $\Omega$, which are
different from the classical asymptotic volume ratio. Notice that our
assumptions on function $f$ focus on $\partial\Omega$ rather than the
whole manifold $M$.

By means of two weighted relative volumes above, we first state Willmore-type
inequalities in weighted manifolds with $\mathrm{Ric}_f\ge 0$.
\begin{theorem}\label{Ricf2}
Let $(M^n,g,e^{-f}dv)$ be a complete noncompact weighted $n$-manifold
with $\mathrm{Ric}_f\ge 0$, and let $\Omega\subset M^n$ be a bounded
open set with smooth boundary $\partial\Omega$. The weighted
mean curvature $H_f$ of $\partial\Omega$ is nonnegative everywhere.

(a) If \eqref{assump2} holds, then
\begin{equation}\label{anWill2}
\int_{\partial\Omega}\left(\frac{H_f}{n-1}\right)^{n-1}e^{-f}d\sigma
\ge\mathrm{RV}_f(\Omega) |\mathbb{S}^{n-1}|.
\end{equation}
Moreover, if $\Omega$ is connected, $\mathrm{RV}_f(\Omega)>0$ and $H_f$ is constant on
$\partial\Omega$, then the equality of \eqref{anWill2} holds if and only
if $\partial\Omega$ is connected and $(M^n\backslash\Omega,g, e^{-f}dv)$ is
isometric to
\begin{equation}\label{anwarprod}
\left([r_0,\infty)\times\partial\Omega,dr^2+(\tfrac{r}{r_0})^2g_{\partial\Omega}\right)
\end{equation}
with $\partial_rf\equiv 0$ along all minimal geodesic segments
from $\partial\Omega$, where $r_0=\left(\frac{|\partial\Omega|_f}{\mathrm{RV}_f(\Omega)|\mathbb{S}^{n-1}|}\right)^{\frac{1}{n-1}}$.

(b) If \eqref{assump2b} holds, then
\begin{equation}\label{anWill2b}
\int_{\partial\Omega}\left(\frac{H_f}{n-1}\right)^{n-1+4k}e^{-f}d\sigma
\ge\overline{\mathrm{RV}}_f(\Omega) |\mathbb{S}^{n-1+4k}|.
\end{equation}
Moreover, if $\Omega$ is connected, $\overline{\mathrm{RV}}_f(\Omega)>0$ and $H_f$ is constant on
$\partial\Omega$, then the equality of \eqref{anWill2b} holds if and only if
$\partial\Omega$ is connected and $(M^n\backslash\Omega,g, e^{-f}dv)$ is
isometric to
\begin{equation}\label{anwarprodb}
\left([r_1,\infty)\times\partial\Omega,dr^2+(\tfrac{r}{r_1})^2g_{\partial\Omega}\right)
\end{equation}
with $ f\equiv 0$ along all minimal geodesic segments from
$\partial\Omega$, where $r_1=\left(\frac{|\partial\Omega|_f}{\overline{\mathrm{RV}}_f(\Omega)|\mathbb{S}^{n-1}|}\right)^{\frac{1}{n-1}}$.
\end{theorem}

\begin{remark}\label{remA}
(i) Condition \eqref{assump2} (or condition \eqref{assump2b}) and $H_f\ge0$ in Theorem
\ref{Ricf2} not only guarantee that weighted mean comparison theorems for hypersurfaces
hold (see Lemma \ref{lem}), but also ensure weighted relative volumes are well defined
(see Proposition \ref{welldef}).

(ii) The connected assumption of $\Omega$ in Theorem \ref{Ricf2} is required in
the rigidity statement, which ensures the connectedness of $\partial\Omega$ and $M^n$
has only one end.

(iii) The constant assumption of $H_f$ in the rigidity statement of Theorem
\ref{Ricf2} may be a technique condition. In manifold case, when the
equality of \eqref{Willineq} holds, there exists a Codazzi equation
connecting the Ricci curvature with the mean curvature, which leads to
the constant property of mean curvature on the boundary (see \cite{Wa}).
In weighted manifolds, when the equality of \eqref{anWill2}
(or \eqref{anWill2b}) occurs, there seems not to be a Codazzi type equation
relating between $\mathrm{Ric}_f$ (or $\mathrm{Ric}_f^m$) and $H_f$, and
we can not get the constant property of $H_f$. Moreover, we do not have
an example where equality is achieved in \eqref{anWill2} (or \eqref{anWill2b}),
but without the conical splitting. Therefore it is interesting to ask if the
constant assumption of $H_f$ can be removed.

(iv) If we give a stronger assumption $|\nabla f|\le a$ (or $|f|\le k$) on
the whole manifold $M^n$, then the Willmore-type inequality \eqref{anWill2}
(or \eqref{anWill2b}) still holds. In this setting, the equality case holds
if and only if $f$ is constant (or $f=0)$ on $M^n$ and rigidity statement is
the same as the manifold case (without the connected assumption of $\Omega$
and the constant assumption of $H_f$).
\end{remark}

On steady gradient Ricci soliton $\mathrm{Ric}_f=0$, we have the identity
$\mathrm{R}+|\nabla f|^2\equiv a$ for some constant $a\ge 0$ (see \cite{Ham}),
where $\mathrm{R}$ is the scalar curvature. By \cite{Cbl}, we have
$\mathrm{R}\ge0$ and hence $|\nabla f|^2\le a$ on $M^n$. So Theorem \ref{Ricf2}
and Remark \ref{remA}(iv) imply that
\begin{corollary}
Let $(M^n,g,e^{-f}dv)$ be a noncompact $n$-dimensional steady gradient Ricci soliton.
Let $\Omega\subset M^n$ be a bounded open set with smooth boundary $\partial\Omega$.
If the weighted mean curvature $H_f$ of $\partial\Omega$ is nonnegative everywhere,
then
\[
\int_{\partial\Omega}\left(\frac{H_f}{n-1}\right)^{n-1}e^{-f}d\sigma
\ge\mathrm{RV}_f(\Omega) |\mathbb{S}^{n-1}|.
\]
Moreover, if $\mathrm{RV}_f(\Omega)>0$, then the equality holds if and only if
$(M^n,g,e^{-f}dv)$ is Ricci flat with the constant function $f$ and
$(M^n\backslash\Omega,g, e^{-f}dv)$ is isometric to
\[
\left([r_0,\infty)\times\partial\Omega,dr^2+(\tfrac{r}{r_0})^2g_{\partial\Omega}\right),
\]
where $r_0=\left(\frac{|\partial\Omega|_f}{\mathrm{RV}_f(\Omega)|\mathbb{S}^{n-1}|}\right)^{\frac{1}{n-1}}$.
In particular, $\partial\Omega$ is a connected totally umbilic submanifold
with constant mean curvature.
\end{corollary}

On $(M^n,g,e^{-f}dv_g)$ with $\mathrm{Ric}_f^m\ge 0$, for a point $p\in M^n$,
we introduce the $m$-weighted asymptotic volume ratio
\[
\mathrm{AVR}^m_f(g):=\lim_{r\to\infty}\frac{|B(p,r)|_f}{|\mathbb{B}^m|r^m}.
\]
By the weighted volume comparison of $\mathrm{Ric}_f^m\ge 0$ (see \cite{BQ}),
$\mathrm{AVR}^m_f(g)$ is well defined and it is independent of the base
point $p$.

When $\mathrm{Ric}_f^m\ge 0$, we can establish another Willmore-type inequality
in weighted manifolds.

\begin{theorem}\label{Ricmf}
Let $(M^n,g,e^{-f}dv)$ be a complete noncompact weighted $n$-manifold
with $\mathrm{Ric}_f^m\ge 0$. Let $\Omega\subset M^n$ be a bounded open
set with smooth boundary $\partial\Omega$. Then
\begin{equation}\label{Will}
\int_{\partial\Omega}\left|\frac{H_f}{m-1}\right|^{m-1}e^{-f}d\sigma
\ge\mathrm{AVR}^m_f(g) |\mathbb{S}^{m-1}|.
\end{equation}
Moreover, if $M^n$ has only one end, $\mathrm{AVR}^m_f(g)>0$ and $H_f$ is
constant on $\partial\Omega$, then the equality of \eqref{Will} implies that
$\partial\Omega$ is connected and $(M^n\backslash\Omega,g, e^{-f}dv)$ is
isometric to
\begin{equation}\label{warprod2}
\left([r_0,\infty)\times\partial\Omega,dr^2+(\tfrac{r}{r_0})^2g_{\partial\Omega}\right),
\end{equation}
where
$r_0=\left(\frac{|\partial\Omega|_f}{\mathrm{AVR}^m_f(g)|\mathbb{S}^{m-1}|}\right)^{\frac{1}{m-1}}$.
\end{theorem}

\begin{remark}
(i) The only one end assumption is due to the rigidity statement of Theorem
\ref{Ricmf}. In manifold case, $\mathrm{Ric}\ge 0$ and $\mathrm{AVR}(g)>0$
suffice to ensure the manifold has only one end by the Cheeger-Gromoll splitting
theorem. However, when $\mathrm{Ric}^m_f\ge 0$ and $\mathrm{AVR}^m_f(g)>0$,
we do not know if the weighted manifold has only one end. It is interesting
to ask if only one end assumption can be removed.

(ii) The constant assumption of $H_f$ in the rigidity statement of Theorem
\ref{Ricmf} is due to the same reason in Remark \ref{remA}(iii).

(iii) When $f$ is constant (and $m=n$), the only one end assumption and
the constant assumption of $H_f$ automatically hold, and Theorem \ref{Ricmf}
returns to Theorem \ref{ACthm}.
\end{remark}

On a complete shrinking gradient Ricci soliton (\emph{shrinker} for short)
$(M^n,g,e^{-f}dv)$, we assume, without loss of generality that (see \cite{Ham})
\begin{equation}\label{norm}
\mathrm{Ric}_f=\frac 12\, g \quad \mathrm{and} \quad  \mathrm{R}+|\nabla f|^2=f.
\end{equation}
Given a base point $p\in M^n$, we consider the classical asymptotic
volume ratio
\[
\mathrm{AVR}(g):=\lim\limits_{r\to\infty}\frac{|B(p,r)|}{|\mathbb{B}^n|r^n}
\]
on shrinkers. From \cite{CLY}, we know that $\mathrm{AVR}(g)$ on shrinkers
always exists and it is independent of the base point $p$. Similar to the argument
of Theorem \ref{Ricf2}, we give a sharp Willmore-like inequality in shrinkers.
\begin{theorem}\label{shri}
Let $(M^n,g,e^{-f}dv)$ be a noncompact $n$-shrinker with \eqref{norm}. Let
$\Omega\subset M^n$ be a bounded open set with smooth boundary $\partial\Omega$.
If the mean curvature $H$ of $\partial\Omega$ is positive everywhere, then
\begin{equation}\label{shwill}
\int_{\partial\Omega}\exp\left\{\frac{(n-1)^2}{4H^2}+f-\frac{n-1}{H}\partial_\nu f\right\}
\cdot\left(\frac{H}{n-1}\right)^{n-1}d\sigma\ge \mathrm{AVR}(g) |\mathbb{S}^{n-1}|,
\end{equation}
where $\partial_\nu f$ denotes the derivative of $f$ in outer unit normal directions
of $\partial\Omega$. Moreover, the equality of \eqref{shwill} holds if and only
if $\partial\Omega$ is a round sphere and $(M^n,g,e^{-f}dv)$ is isometric to the
Gaussian shrinker $(\mathbb{R}^n, \delta_{ij},e^{-|x|^2/4}dv)$.
\end{theorem}

\begin{remark}
We would like to point out that the equality of \eqref{shwill} implies that the
scalar curvature of shrinker vanishes; see Section \ref{remshri} for the proof.
Combining it with \cite{PRS}, such shrinker must be isometric to the Gaussian
shrinker.
\end{remark}

As applications of Theorems \ref{Ricf2} and \ref{Ricmf}, when a weighted
manifold admits a critical set of weighted isoperimetric functional, we
can get some sharp isoperimetric type inequalities; see Theorems \ref{iso1}
and \ref{iso2} in Section \ref{isoin}. These results can be regarded as
weighted versions of a sharp isoperimetric inequality in manifolds proved
by Ros \cite{Ro}. We remark that from Theorem \ref{shri}, there seems to
be some obstacle to deduce a sharp isoperimetric inequality in shrinkers.
This is because the Willmore-like inequality \eqref{shwill} contains an
exponential term which is obviously different from the classical Willmore
inequality and we do not obtain a isoperimetric type inequality.

It is well-known that the Willmore inequality \eqref{Willineq} leads to sharp
global isoperimetric inequalities when the manifold has $\mathrm{Ric}\ge0$ and
Euclidean volume growth. The $3$-dimensional case was proved by Agostiniani,
Fogagnolo and Mazzieri \cite{AFM} by using Huisken's mean curvature flows;
the $n$-dimensional ($3\le n\le 7$) case was proved by Fogagnolo and Mazzieri
\cite{FM} by using an exhaustion of outward minimising sets and Kleiner's
arguments \cite{Kle}. As pointed out by the anonymous referee, inspired by the
manifold case, it is natural to ask if our Willmore type inequalities (Theorems
\ref{Ricf2}, \ref{Ricmf} and \ref{shri}) could yield sharp weighted isoperimetric
inequalities by using Fogagnolo-Mazzieri's argument \cite{FM}. We notice that
Johne \cite{Jo} has recently obtained a sharp weighted isoperimetric type inequality
in weighted manifolds when $\mathrm{Ric}_f^m\ge 0$ ($m<\infty$) by using ABP-method,
generalizing Brendle's manifold case \cite{Br2}. Balogh and Krist\'aly \cite{BK}
have generalized isoperimetric inequalities to a nonsmooth setting. They have
obtained a sharp isoperimetric inequality in $CD(0,N)$ metric measure spaces
by using optimal mass transport theory. But up to now, to the authors' knowledge,
there has not been any sharp weighted isoperimetric type inequality in weighted
manifolds when $\mathrm{Ric}_f\ge 0$. To solve this problem, inspired by
Fogagnolo-Mazzieri's work \cite{FM}, we may need to consider many issues
in weighted manifolds, for example, the existence of a weighted mean-convex
exhaustion, the existence and regularity of weighted constrained isoperimetric
sets, and the fundamental relation on the weighted mean curvature of boundaries
for a bounded subset and its weighted constrained isoperimetric set, etc.
These issues and eventual sharp weighted isoperimetric inequalities via Willmore
type inequalities in weighted manifolds will be expected to be studied in the
future.

The rest of paper is organized as follows. In Section \ref{comp}, we prove
some weighted mean curvature comparison theorems for hypersurfaces in weighted
manifolds. Using them, we prove that two weighted relative volumes defined in
introduction are well defined. We also prove some rigidity results for compact
weighted manifolds with boundaries. In Section \ref{willm1}, we apply our new
comparison results to prove Theorem \ref{Ricf2}. In Section \ref{Rmfr}, we
apply a similar argument of Theorem \ref{Ricf2} to prove Theorem \ref{Ricmf}.
In Section \ref{remshri}, we give a proof of Theorem \ref{shri}. In Section
\ref{isoin}, we apply Theorems \ref{Ricf2} and \ref{Ricmf} to prove sharp
isoperimetric type inequalities in weighted manifolds.

%55555555555555555555555555555555555555555555555555555555555555555555555555555555555555555555555555555555555555555555555555555555555555555555555555555555555555

\section{Comparison theorems}\label{comp}
In this section, we give some weighted mean curvature comparison theorems for
hypersurfaces under the nonnegative Bakry-\'Emery Ricci curvature. As
applications, we show that weighted relative volumes in introduction are
well defined. Besides, we recall a generalized Reilly formula in weighted
manifolds. Using this Reilly formula, we prove a rigidity result and a
Heintze-Karcher type inequality for compact weighted manifolds with boundaries.

On $(M^n,g,e^{-f}dv)$, let $\Omega\subset M^n$ be a bounded open set with
smooth boundary $\partial\Omega$. Throughout this paper, we denote by $\nabla$
and $\Delta$ the gradient and the Laplacian on $\Omega$ of $(M^n,g)$
respectively, while by $\nabla_{\partial\Omega}$ and $\Delta_{\partial\Omega}$
the gradient and the Laplacian on $\partial\Omega$ respectively. We let
$\nu$ denote the outer unit normal vector of $\partial\Omega$. Then the second
fundamental form of $\partial\Omega$ is defined by
\[
h(X,Y)=g(\nabla_X\nu, Y),
\]
where $X,Y\in T(\partial\Omega)$. The trace of the second fundamental
$h$, i.e., $H=\mathrm{Trace}_gh$, denotes the mean curvature of
$\partial\Omega$, and the corresponding weighted mean curvature is
defined by
\[
H_f:=H-g(\nabla f,\nu).
\]
The weighted mean curvature $H_f$ often appears in the critical point of
weighted area functional of the submanifold $\partial\Omega$.
We say that $\partial\Omega$ is weighted minimal if $H_f\equiv0$.

For a fixed point $p\in\partial\Omega$, let $\gamma_p(t)=\exp_pt\nu(p)$
be the normal geodesic with initial velocity $\nu(p)$. We set
\[
\tau(p)=\sup\left\{ l>0\,|\, \gamma_p \, \text{is minimizing the distance to $p\in\partial\Omega$ on}\, [0,l] \right\}\in(0,\infty].
\]
Then $\tau$ is continuous on $\partial\Omega$ and the focus locus
\[
C(\partial\Omega):=\{\exp_p\tau(p)\nu(p)|\tau(p)<\infty\}
\]
is a closed set of measure zero in $M$. Let
\[
E=\{(r,p)\in[0,\infty)\times\partial\Omega|r<\tau(p)\}
\]
and the exponential map $\Phi: E\to(M\backslash\Omega)\backslash C(\partial\Omega)$
defined by $\Phi(r,p)=\exp_pr\nu(p)$ is a diffeomorphism. On $E$, the pull back
of weighted volume form in polar coordinates is
\[
e^{-f}dv=\mathcal{A}_f(r,p)drd\sigma(p),
\]
where $\mathcal{A}_f(r,p)=e^{-f}\mathcal{A}(r,p)$ and $\mathcal{A}(r,p)$
is the Riemannian volume element in geodesic polar coordinates. Customarily,
we understand $r$ as the distance function to $\partial\Omega$ and it is smooth
on $M\backslash\Omega$ away from $C(\partial\Omega)$.

Below we shall prove mean curvature comparison theorems for hypersurfaces
in weighted manifolds with $\mathrm{Ric}_f\ge 0$, extending the manifold case. Compared
with the manifold version, a obvious difficultly is that our proof needs to deal
with some extra terms of $f$. Fortunately, we can employ Wei-Wylie's proof trick
\cite{WW} to overcome this difficultly.

\begin{lemma}\label{lem}
Let $(M^n,g,e^{-f}dv)$ be a complete noncompact weighted manifold with
$\mathrm{Ric}_f\ge 0$. Let $\Omega\subset M^n$ be a bounded open set with
smooth boundary $\partial\Omega$. The weighted mean curvature $H_f$ of
$\partial\Omega$ is nonnegative everywhere. Fix a point $p\in \partial\Omega$.

(a) If $\partial_rf\ge -a$ for some constant $a\ge 0$ along a minimal
geodesic segment from $p$, then $r(x)=d(p,x)$ satisfies
\[
\Delta_f r\le\frac{(n-1)H_f(p)}{n-1+H_f(p)r}+a-\frac{(n-1)^2a}{\left[n-1+H_f(p)r\right]^2}
\]
for $r\in[0,\tau(p))$ along that minimal geodesic segment from $p$.

(b) If $|f|\le k$ for some constant $k\ge 0$ along a minimal
geodesic segment from $p$, then $r(x)=d(p,x)$ satisfies
\[
\Delta_f r\le\frac{(n-1+4k)H_f(p)}{n-1+H_f(p)r}
\]
for $r\in[0,\tau(p))$ along that minimal geodesic segment from $p$.
\end{lemma}
\begin{proof}[Proof of Lemma~\ref{lem}]
For a fixed point $p\in \partial\Omega$, using $|\nabla r|=1$ almost everywhere
in the Bochner formula, we get that
\begin{equation}\label{ricca}
\begin{aligned}
0=\frac{1}{2}\Delta|\nabla r|^2&=|\mathrm{Hess}\,r|^2
+\frac{\partial}{\partial r}\left(\Delta r\right)+{\rm Ric}(\nabla r,\nabla r)\\
&\ge\frac{(\Delta r)^2}{n-1}+\frac{\partial}{\partial r}\left(\Delta r\right)
+{\rm Ric}(\nabla r,\nabla r)
\end{aligned}
\end{equation}
for $r\in[0,\tau(p))$, where we used the Cauchy-Schwarz inequality in the second
inequality. Since $\mathrm{Ric}_f\ge 0$, then
\[
\frac{\partial}{\partial r}(\Delta r)+\frac{(\Delta r)^2}{n-1}\le f''(r),
\]
for $r\in[0,\tau(p))$, where $f''(r):=\mathrm{Hess} f(\partial r,\partial r)=\frac{d^2}{dr^2}(f\circ \gamma)(r)$.
This inequality can be written as
\[
\frac{\tfrac{\partial}{\partial r}[(n-1+H_f(p)r)^2\Delta r]}{(n-1+H_f(p)r)^2}
+\frac{1}{n-1}\left[\Delta r-\frac{(n-1)H_f(p)}{n-1+H_f(p)r}\right]^2
\le\frac{(n-1)H_f(p)^2}{(n-1+H_f(p)r)^2}+f''(r)
\]
for $r\in[0,\tau(p))$. Discarding the square term on the left hand side, we have
\[
\frac{\partial}{\partial r}\left[(n-1+H_f(p)r)^2\Delta r\right]
\le(n-1)H_f(p)^2+f''(r)(n-1+H_f(p)r)^2
\]
for $r\in[0,\tau(p))$. Integrating this inequality and using the initial condition
$\Delta r|_{r=0}=H(p)$, we get
\begin{equation*}
\begin{aligned}
(n-1+H_f(p)r)^2\Delta r-(n-1)^2H(p)&\le(n-1)H_f(p)^2r+(n-1+H_f(p)t)^2f'(t){\Big|}^r_0\\
&\quad-\int^r_0f'(t)\left[(n-1+H_f(p)t)^2\right]'dt
\end{aligned}
\end{equation*}
for $r\in[0,\tau(p))$. Rearranging some terms, we have
\begin{equation}
\begin{aligned}\label{meagen}
\left[n-1+H_f(p)r\right]^2\Delta_f r&\le (n-1)H_f(p)\left[n-1+H_f(p)r\right]\\
&\quad-\int^r_0f'(t)\left[(n-1+H_f(p)t)^2\right]'dt
\end{aligned}
\end{equation}
for $r\in[0,\tau(p))$.

Case (a): If $f'(t)\ge -a$ and $H_f(p)\ge0$, then
$\big[(n-1+H_f(p)t)^2\big]'\ge0$ and hence
\begin{equation}
\begin{aligned}\label{meanva}
-\int^r_0f'(t)\left[(n-1+H_f(p)t)^2\right]'dt
&\le a\int^r_0\left[(n-1+H_f(p)t)^2\right]'dt\\
&=a \left[n-1+H_f(p)r\right]^2-(n-1)^2a.
\end{aligned}
\end{equation}
Substituting this into \eqref{meagen} yields the conclusion of Case (a).

Case (b): If $|f|\le k$ and $H_f(p)\ge0$, by the integration by parts,
we have that
\begin{equation*}
\begin{aligned}
-&\int^r_0f'(t)\big[(n-1+H_f(p)t)^2\big]'dt\\
&=-2H_f(p)\int^r_0[n-1+H_f(p)t]df(t)\\
&=-2\big[n-1+H_f(p)r\big]H_f(p)f(r)+2(n-1)H_f(p)f(p)+2H_f(p)^2\int^r_0f(t)dt\\
&\le 2kH_f(p)\big[n-1+H_f(p)r\big]+2(n-1)kH_f(p)+2kH_f(p)^2r\\
&=4kH_f(p)\big[n-1+H_f(p)r\big].
\end{aligned}
\end{equation*}
Substituting this into \eqref{meagen} yields the desired estimate of Case (b).
\end{proof}

Next, we shall give another weighted mean curvature comparison theorems for
tubular hypersurfaces in weighted manifolds with $\mathrm{Ric}^m_f\ge 0$
(without any assumption on $f$), also extending the manifold case.

\begin{lemma}\label{lemv2}
Let $(M^n,g,e^{-f}dv)$ be a complete noncompact weighted manifold with
$\mathrm{Ric}^m_f\ge 0$. Let $\Omega\subset M^n$ be a bounded open set with
smooth boundary $\partial\Omega$. For any a point $p\in \partial\Omega$,
the distance function $r(x)=d(p,x)$ satisfies
\[
\Delta_f r\le\frac{(m-1)H_f(p)}{m-1+H_f(p)r}
\]
for $r\in[0,\tau(p))$, where $\tau(p)\le\frac{m-1}{H_f^{-}(p)}$
and $H_f^{-}(p):=\max\{-H_f(p),\,0\}$.
\end{lemma}
\begin{proof}[Proof of Lemma \ref{lemv2}]
Applying the generalized Bochner formula to the distance function $r(x)=d(p,x)$,
we get that
\begin{align*}
0=\frac{1}{2}\Delta_f|\nabla r|^2&=|\mathrm{Hess}\,r|^2
+\langle \nabla r,\nabla\Delta_f r\rangle+\mathrm{Ric}_f(\nabla r,\nabla r)\\
&\ge\frac{(\Delta r)^2}{n-1}+\frac{\partial}{\partial r}(\Delta_f r)
+\mathrm{Ric}_f(\nabla r,\nabla r)\\
&=\frac{(\Delta_f r+\langle\nabla f,\nabla r\rangle)^2}{n-1}
+\frac{\partial}{\partial r}(\Delta_f r)+\mathrm{Ric}_f(\nabla r,\nabla r)\\
&\ge\frac{(\Delta_f r)^2}{m-1}+\frac{\partial}{\partial r}(\Delta_f r)
+\mathrm{Ric}^m_f(\nabla r,\nabla r)\\
&\ge\frac{(\Delta_f r)^2}{m-1}+\frac{\partial}{\partial r}(\Delta_f r)
\end{align*}
for $r\in[0,\tau(p))$, where we used the Cauchy-Schwarz inequality in the
second inequality and we used $\mathrm{Ric}^m_f\ge 0$ in the third inequality.
Considering the initial value $\Delta_f r|_{r=0}=H_f(p)$, we solve the above
inequality and give the desired result.
\end{proof}

In the following, we show that Lemma \ref{lem} implies that weighted relative
volumes $\mathrm{RV}_f(\Omega)$ and $\overline{\mathrm{RV}}_f(\Omega)$ given
in introduction are both well defined.

\begin{proposition}\label{welldef}
On a weighted manifold $(M^n,g,e^{-f}dv_g)$ with $\mathrm{Ric}_f\ge 0$, let
$\Omega\subset M^n$ be a bounded open set with smooth boundary $\partial\Omega$.
The weighted mean curvature $H_f$ of $\partial\Omega$ is nonnegative everywhere.

(a) If \eqref{assump2} holds, then
\[
\lim_{r\to\infty}\frac{\mathrm{Vol}_f\{x\in M|d(x,\Omega)<r\}}{m(r)},
\]
where $m(r):=n |\mathbb{B}^n|\int^r_0e^{at}t^{n-1}dt$, exists. Hence
$\mathrm{RV}_f(\Omega)$ is well defined.

(b) If \eqref{assump2b} holds, then
\[
\lim_{r\to\infty}\frac{\mathrm{Vol}_f\{x\in M|d(x,\Omega)<r\}}
{|\mathbb{B}^{n+4k}|r^{n+4k}},
\]
exists. Hence $\overline{\mathrm{RV}}_f(\Omega)$ is well defined.
\end{proposition}
\begin{proof}[Proof of Proposition \ref{welldef}]
We set $\bar{\mathcal{A}}_f(r,p)=\mathcal{A}_f(r,p)$ for $r<\tau(p)$
and $\bar{\mathcal{A}}_f(r,p)=0$ for $r\ge \tau(p)$.

At first we prove part (a). We claim that
\[
\kappa_1(p):=\lim_{r\to+\infty}\frac{\bar{\mathcal{A}}_f(r,p)}{m'(r)}
\]
exists, where $m(r):=n |\mathbb{B}^n|\int^r_0e^{at}t^{n-1}dt$. In fact, for
the case $r\ge\tau(p)$, the claim is obvious. We only consider the case $r<\tau(p)$.
Recalling that $\Delta r=(\ln\mathcal{A})^{\prime}$ and
$\mathcal{A}_f(r,p)=e^{-f}\mathcal{A}(r,p)$, then
$\Delta_f r=(\ln\mathcal{A}_f)^{\prime}$. By Lemma \ref{lem}(a), we see that
\begin{equation}
\begin{aligned}\label{compdic}
\frac{\mathcal{A}^{\prime}}{\mathcal{A}}&\le\frac{(n-1)H_f(p)}{n-1+H_f(p)r}+a
-\frac{(n-1)^2a}{\left[n-1+H_f(p)r\right]^2}\\
&\le\frac{(n-1)H_f(p)}{n-1+H_f(p)r}+a
\end{aligned}
\end{equation}
for $r<\tau(p)$, where we discarded the third non-positive term in the second inequality.
Hence, for each $p\in\partial\Omega$,
\[
\theta_f(r,p):=\frac{\mathcal{A}_f(r,p)}{e^{ar}\left(1+\frac{H_f(p)}{n-1}r\right)^{n-1}}
\]
is non-increasing and bounded on $[0,\tau(p))$, which implies that
$\lim_{r\to\infty}\theta_f(r,p)$ exists. We also see that
\[
\lim_{r\to\infty}\frac{e^{ar}\left(1+\frac{H_f(p)}{n-1}r\right)^{n-1}}{m'(r)}
=\lim_{r\to\infty}\frac{e^{ar}\left(1+\frac{H_f(p)}{n-1}r\right)^{n-1}}{n|\mathbb{B}^n|e^{ar}r^{n-1}}
=\frac{1}{n|\mathbb{B}^n|}\left(\frac{H_f(p)}{n-1}\right)^{n-1}
\]
exists. Combining two aspects above indicates that
$\lim_{r\to\infty}\frac{\mathcal{A}_f(r,p)}{m'(r)}$ exists. Therefore
the claim follows. From this claim, we see that there exists a positive
constant $c_1$ such that $\bar{\mathcal{A}}_f(r,p)\le c_1m'(r)$. Thus,
\[
\frac{\int^r_0\bar{\mathcal{A}}_f(t,p)dt}{m(r)}\le\frac{\int^r_0c_1m'(t)dt}{m(r)}=c_1.
\]
Using this upper bound, by the Lebesgue dominated convergence theorem and the
L'Hopital rule, we have that
\begin{align*}
\mathrm{RV}_f(\Omega)&:=\lim_{r\to\infty}\frac{\mathrm{Vol}_f\{x\in M|d(x,\Omega)<r\}}{m(r)}\\
&=\lim_{r\to\infty}\frac{\int_{\Sigma}\int^r_0\bar{\mathcal{A}}_f(t,p)dtd\sigma(p)}{m(r)}\\
&=\int_{\Sigma}\lim_{r\to\infty}\frac{\bar{\mathcal{A}}_f(r,p)}{m'(r)}d\sigma(p)\\
&=\int_{\Sigma}\kappa_1(p)d\sigma(p).
\end{align*}
This indicates that $\mathrm{RV}_f(\Omega)$ is well defined.

We then prove part (b). In this case, we claim that
\[
\kappa_2(p):=\lim_{r\to\infty}\frac{\bar{\mathcal{A}}_f(r,p)}{(n+4k)|\mathbb{B}^{n+4k}|r^{n-1+4k}}
\]
exists. Indeed, for the case $r\ge\tau(p)$, it is obvious. We only consider
the case $r<\tau(p)$. By Lemma \ref{lem}(b), for each $p\in\partial\Omega$,
we see that
\[
\overline{\theta}_f(r,p):=\frac{\bar{\mathcal{A}}_f(r,p)}{\left(1+\frac{H_f(p)}{n-1}r\right)^{n-1+4k}}
\]
is non-increasing and bounded on $[0,\tau(p))$, which implies that
$\lim_{r\to\infty}\overline{\theta}_f(r,p)$ exists. We also see that
\[
\lim_{r\to\infty}\frac{\left(1+\frac{H_f(p)}{n-1}r\right)^{n-1+4k}}{(n+4k)|\mathbb{B}^{n+4k}|r^{n-1+4k}}
=\frac{\left(\frac{H_f(p)}{n-1}\right)^{n-1+4k}}{(n+4k)|\mathbb{B}^{n+4k}|}
\]
exists. Combining the above two aspects gives that
$\lim_{r\to\infty}\frac{\mathcal{A}_f(r,p)}{(n+4k)|\mathbb{B}^{n+4k}|r^{n-1+4k}}$
exists. Hence the claim follows. By this claim, there exists a positive constant
$c_2$ such that $\bar{\mathcal{A}}_f(r,p)\le c_2(n+4k)|\mathbb{B}^{n+4k}|r^{n-1+4k}$.
Thus,
\[
\frac{\int^r_0\bar{\mathcal{A}}_f(t,p)dt}{|\mathbb{B}^{n+4k}|r^{n+4k}}
\le\frac{\int^r_0c_2(n+4k)|\mathbb{B}^{n+4k}|t^{n-1+4k}dt}{|\mathbb{B}^{n+4k}|r^{n+4k}}=c_2.
\]
Using this upper bound, by the Lebesgue dominated convergence theorem and the
L'Hopital rule, we have that
\begin{align*}
\overline{\mathrm{RV}}_f(\Omega)&:=\lim_{r\to\infty}
\frac{\mathrm{Vol}_f\{x\in M|d(x,\Omega)<r\}}{|\mathbb{B}^{n+4k}|r^{n+4k}}\\
&=\lim_{r\to\infty}\frac{\int_{\Sigma}\int^r_0\bar{\mathcal{A}}_f(t,p)dtd\sigma(p)}{|\mathbb{B}^{n+4k}|r^{n+4k}}\\
&=\int_{\Sigma}\lim_{r\to\infty}\frac{\bar{\mathcal{A}}_f(r,p)}{(n+4k)|\mathbb{B}^{n+4k}|r^{n-1+4k}}d\sigma(p)\\
&=\int_{\Sigma}\kappa_2(p)d\sigma(p).
\end{align*}
This shows that $\overline{\mathrm{RV}}_f(\Omega)$ is well defined.
\end{proof}

Ichida \cite{Ic} and  Kasue \cite{Kas} proved that if a compact connected
manifold $M^n$ with mean convex boundary $\Sigma$ has nonnegative Ricci
curvature, then $\Sigma$ has at most two components; moreover if $\Sigma$
has two components, then $M^n$ is isometric to $N\times [0,l]$ for some
connected compact $(n-1)$-manifold $N$ and some constant $l>0$. See also an
alternative proof in \cite{HW}. Now we generalize this result to the weighted
manifold. To achieve it, we need a generalized Reilly formula \cite{KM},
which states that for any $u\in C^\infty(M)$,
\begin{equation}
\begin{aligned}\label{reif}
\int_M(\Delta_fu)^2d\mu&-\int_M|\mathrm{Hess}\,u|^2 d\mu
-\int_M\mathrm{Ric}_f(\nabla u, \nabla u)d\mu\\
&=2\int_{\Sigma}u_\nu({\Delta_f}_{\Sigma}z)d\mu_{\Sigma}
+\int_{\Sigma} H_f(u_\nu)^2 d\mu_{\Sigma}
+\int_{\Sigma} h(\nabla_{\partial\Omega}z,\nabla_{\Sigma}z)d\mu_{\Sigma},
\end{aligned}
\end{equation}
where $z=u|_{\Sigma}$, $d\mu=e^{-f}dv$ and $d\mu_{\Sigma}=d\mu|_{\Sigma}$.
By this formula, we can apply Hang-Wang's argument \cite{HW} to get the
following result, which will be used in the proof of Theorem \ref{Ricf2}.
\begin{proposition}\label{num}
Let $(M^n,g,e^{-f}dv)$ be a compact connected weighted $n$-manifold
with $\mathrm{Ric}_f\ge 0$. If the boundary $\Sigma$ of $M^n$ has
$H_f\ge0$, then $\Sigma$ has at most two components. Moreover, if
$\Sigma$ has two components, then $\Sigma$ is totally geodesic and $M^n$
is isometric to
\[
N\times [0,l]
\]
for some constant $l>0$, where $N$ is a connected compact $(n-1)$-manifold.
\end{proposition}

\begin{proof}[Proof of Proposition \ref{num}]
Assume that $\Sigma$ is not connected. Fixing a component $\Sigma_0$ of $\Sigma$,
we solve the Dirichlet problem
\begin{equation*}
\left\{
\begin{aligned}
\Delta_f u=0 \quad & \text{in}\quad M,\\
u|_{\Sigma_0}=0 \quad & \text{and}\quad u|_{\Sigma\backslash\Sigma_0}=1.
\end{aligned}
\right.
\end{equation*}
Applying \eqref{reif} to function $u$, we have
\[
-\int_M|\mathrm{Hess}\,u|^2 d\mu=\int_M\mathrm{Ric}_f(\nabla u, \nabla u)d\mu
+\int_{\Sigma} H_f(u_\nu)^2 d\mu_{\Sigma}.
\]
Since $\mathrm{Ric}_f\ge 0$ and $H_f\ge0$, then $\mathrm{Hess}\,u=0$,
which implies $|\nabla u|=k$ for some constant $k>0$. Since
$\nabla u=-k\nu$ on $\Sigma_0$ and $\nabla u=k\nu$ on
$\Sigma\backslash\Sigma_0$, we have
\[
\nabla_X\nu=0
\]
for any $X\in T\Sigma$, that is, $\Sigma$ is totally geodesic. If we consider
the flow generated by $\nabla u/k$, then the flow starts from $\Sigma_0$ and reaches
$\Sigma\backslash\Sigma_0$ at time $1/k$ and hence $\Sigma$ has exactly
two components. Indeed the flow lines are just geodesics.
If we fix local coordinates $\{x_1,\cdots,x_{n-1}\}$ on $\Sigma_0$
and let $r=u/k$, then
\[
g=dr^2+g_{ij}(r,x)dx^idx^j.
\]
Noticing $\mathrm{Hess}\,r=0$, so $\partial_rg_{ij}(r,x)=0$. Therefore
$M^n$ is isometric to $\Sigma_0\times[0,1/k]$.
\end{proof}

The Reilly formula \eqref{reif} can be also used to prove the weighted
Heintze-Karcher inequality, extending the manifold case in \cite{HK, Ro}.
We remark that the same inequality has been obtained in \cite{BCP}.
For the discussion convenience in the proof of Theorem \ref{iso2} below,
we give a detailed proof.
\begin{proposition}[\cite{BCP}]\label{inteinequ}
Let $(M^n,g,e^{-f}dv)$ be a compact weighted $n$-manifold with smooth boundary
$\Sigma$ satisfying $\mathrm{Ric}^m_f\ge0$. If the weighted
mean curvature $H_f$ of $\Sigma$ is positive
everywhere, then
\begin{eqnarray}\label{ros}
\int_{\Sigma} \frac{1}{H_f}d\mu_{\Sigma}\ge\frac{m}{m-1}|M^n|_f.
\end{eqnarray}
Moreover, the equality of \eqref{ros} holds if and only if $m=n$
($f$ is constant) and $M^n$ is isometric to an Euclidean $n$-ball.
\end{proposition}
\begin{proof}[Proof of Proposition \ref{inteinequ}]
The proof is analogous to Theorem 1 in \cite{Ro}, which originates from
\cite{Re}. Let $u\in C^\infty(M^n)$ be a smooth (non-constant) function
of the Dirichlet problem
\begin{equation}\label{Diri}
\left\{
\begin{aligned}
\Delta_f u=1 \quad & \text{in}\quad M^n,\\
u|_{\Sigma}=z=0 \quad & \text{on}\quad \Sigma.
\end{aligned}
\right.
\end{equation}
Using \eqref{reif}, we have
\begin{equation}\label{relfor}
\int_M(\Delta_fu)^2 d\mu=\int_M|\mathrm{Hess}\,u|^2d\mu
+\int_M\mathrm{Ric}_f(\nabla u, \nabla u)d\mu
+\int_{\Sigma} H_f(u_\nu)^2 d\mu_{\Sigma}.
\end{equation}
Using the Cauchy-Schwarz inequality, we have that
\begin{equation}
\begin{aligned}\label{Hess}
|\mathrm{Hess}\,u|^2&\ge\frac{1}{n}(\Delta u)^2\\
&=\frac{1}{n}(\Delta_f u+\langle\nabla f,\nabla u\rangle)^2\\
&\ge\frac{1}{n}\left[\frac{|\Delta_f u|^2}{1+\tfrac{m-n}{n}}
-\frac{\langle\nabla f,\nabla u\rangle^2}{\tfrac{m-n}{n}}\right]\\
&=\frac{(\Delta_f u)^2}{m}-\frac{\langle\nabla f,\nabla u\rangle^2}{m-n},
\end{aligned}
\end{equation}
where $m\ge n$ and $f$ is constant when $m=n$. Substituting this into
\eqref{relfor} and combining the fact $\Delta_fu=1$, we get
\[
\left(1-\frac 1m\right)|M^n|_f\ge\mathrm{Ric}^m_f(\nabla u, \nabla u)d\mu
+\int_{\Sigma} H_f(u_\nu)^2 d\mu_{\Sigma}.
\]
Since $\mathrm{Ric}^m_f\ge 0$, we further have
\begin{equation}\label{relfor2}
\left(1-\frac 1m\right)|M^n|_f\ge\int_{\Sigma} H_f(u_\nu)^2 d\mu_{\Sigma}.
\end{equation}
On the other hand, by the divergence theorem, we observe
\[
|M^n|_f=\int_M(\Delta_fu)d\mu=-\int_{\Sigma}u_{\nu}\,d\mu_{\Sigma}
\]
and hence we have
\[
|M^n|^2_f=\left(\int_{\Sigma}u_{\nu}\,d\mu_{\Sigma}\right)^2
\le\int_{\Sigma}H_f (u_{\nu})^2 \,d\mu_{\Sigma}
\cdot\int_{\Sigma}H_f^{-1}\,d\mu_{\Sigma},
\]
where we used the Cauchy-Schwarz inequality. Combining this with
\eqref{relfor2} gives \eqref{ros}.

Now we discuss the equality case of \eqref{ros}. In this case, we can show that
$m=n$ and $f$ is constant. Hence the equality case reduces to the equality
case of Theorem 1 in \cite{Ro} and the rigidity statement follows.

To prove $m=n$, we assume by contradiction that $m>n$. The equality of \eqref{ros} implies the
equality of \eqref{Hess} and it gives that
\[
\mathrm{Hess}\, u=\frac{\Delta u}{n}g\quad\text{and}\quad
\Delta_fu=-\frac{m}{m-n}\langle\nabla f,\nabla u\rangle
\]
in $M^n$. Combining the fact $\Delta_fu=1$ from \eqref{Diri}, we
further have that
\[
\mathrm{Hess}\,u=\frac{1}{m}g\quad\text{and}\quad \Delta u+\frac{n}{m-n}\langle\nabla f,\nabla u\rangle=0
\]
in $M^n$. The latter equality is equivalent to
\[
e^{-\frac{nf}{m-n}}\nabla\left(e^{\frac{nf}{m-n}}\nabla u\right)=0
\]
in $M^n$. Multiplying the factor $ue^{\frac{nf}{m-n}}$ in the above equality
and then integrating it over the compact manifold $M^n$ with respect to the
Riemannian measure,
we have that
\begin{align*}
0&=-\int_{M^n}u\nabla\left(e^{\frac{nf}{m-n}}\nabla u\right)dv\\
&=\int_{M^n}e^{\frac{nf}{m-n}}|\nabla u|^2dv
-\int_{\Sigma}u\left\langle e^{\frac{nf}{m-n}}\nabla u,\nu\right\rangle dv\\
&=\int_{M^n}e^{\frac{nf}{m-n}}|\nabla u|^2dv,
\end{align*}
where we used the divergence theorem in the second equality and $u=0$
on $\Sigma$ in the third equality. This gives $|\nabla u|=0$
and hence $u$ is constant in $M^n$. This contradicts with a fact that
$u$ is not constant due to \eqref{Diri}. Therefore we must have $m=n$ and
$f$ is constant.
\end{proof}

The argument of Proposition \ref{inteinequ} can be extended to
the setting of $\mathrm{Ric}_f\ge0$ if we drop the first nonnegative
term of the right hand side of \eqref{relfor}.
\begin{proposition}\label{inGRS}
Let $(M^n,g,e^{-f}dv)$ be an $n$-dimensional compact set with smooth
boundary $\Sigma$ satisfying $\mathrm{Ric}_f\ge0$. If the weighted
mean curvature $H_f$ of $\Sigma$
is positive everywhere, then
\begin{equation}\label{weakHK}
\int_{\Sigma} \frac{1}{H_f}d\mu_{\Sigma}\ge |M^n|_f.
\end{equation}
\end{proposition}
Proposition \ref{inGRS} can be regarded as another Heintze-Karcher type inequality
for weighted manifolds, which is suitable to non-expanding gradient Ricci solitons.
From Theorem 1 in \cite{Ro} (or Proposition \ref{inteinequ}), we know that
\eqref{weakHK} is not sharp when $f$ is constant, but it is possible sharp for a
general function $f$. At present we do not know how to obtain a rigidity statement
from Proposition \ref{inGRS}. We mention that many Heintze-Karcher type inequalities
were investigated in \cite{Br,LX,QX,WX} and references therein.

%55555555555555555555555555555555555555555555555555555555555555555555555555555555555555555555555555555555555555555555555555555555555555555555555555555555555555

\section{Willmore-type inequality for $\mathrm{Ric}_f\ge 0$}\label{willm1}

In this section we study Willmore-type inequalities in weighted manifolds
with $\mathrm{Ric}_f\ge 0$. We shall apply Lemma \ref{lem} to prove Theorem
\ref{Ricf2} by using Wang's argument \cite{Wa}. We first prove the case
(a) of Theorem \ref{Ricf2}.
\begin{proof}[Proof of Theorem \ref{Ricf2}(a)]
We start to prove \eqref{anWill2}. Let $\Omega\subset M^n$ be a bounded
open set with smooth boundary $\Sigma:=\partial\Omega$. We first discuss
the case when $\Omega$ has no hole, that is $M^n\backslash\Omega$ has
no bounded component. From the proof of Proposition \ref{welldef}(a),
we get that
\[
\theta_f(r,p):=\frac{\mathcal{A}_f(r,p)}{e^{ar}\left(
1+\frac{H_f(p)}{n-1}r\right)^{n-1}}
\]
is non-increasing in $r$ on $[0,\tau(p))$. This monotonicity gives that
\[
\mathcal{A}_f(r,p)\le e^{ar-f(p)}\left(1+\frac{H_f(p)}{n-1}r\right)^{n-1}
\]
for all $r<\tau(p)$, where we used $\theta_f(0,p)=e^{-f(p)}$. For any
$R>0$, we apply the upper bound to estimate that
\begin{equation}
\begin{aligned}\label{weighetvol1}
\mathrm{Vol}_f\{x\in M^n|d(x,\Omega)<R\}-|\Omega|_f=&
\int_{\Sigma}\int_0^{\min\{R,\tau(p)\}}\mathcal{A}_f(r,p)drd\sigma(p)\\
\le&\int_{\Sigma}\int_0^{\min\{R,\tau(p)\}}e^{ar-f(p)}
\left(1+\frac{H_f(p)}{n-1}r\right)^{n-1}drd\sigma(p)\\
\le&\int_{\Sigma}\int_0^R e^{ar-f(p)}\left(1+\frac{H_f(p)}{n-1}r\right)^{n-1}drd\sigma(p).
\end{aligned}
\end{equation}
Dividing both sides by $m(R):=n |\mathbb{B}^n|\int^R_0e^{at}t^{n-1}dt$
and letting $R\rightarrow\infty$ yields
\begin{align*}
\mathrm{RV}_f(\Omega)&\le\lim_{R\to\infty}\frac{\int_{\Sigma}\int_0^R e^{ar-f(p)}
\left(1+\frac{H_f(p)}{n-1}r\right)^{n-1}drd\sigma(p)}{m(R)}\\
&=\lim_{R\to\infty}\frac{\int_{\Sigma}e^{aR-f(p)}
\left(1+\frac{H_f(p)}{n-1}R\right)^{n-1}d\sigma(p)}{n|\mathbb{B}^n|e^{aR}R^{n-1}}\\
&=\frac{1}{|\mathbb{S}^{n-1}|
}\int_{\Sigma}\left(\frac{H_f}{n-1}\right)^{n-1}e^{-f}d\sigma,
\end{align*}
where we used the L'Hopital rule in the second equality and
$n|\mathbb{B}^n|=|\mathbb{S}^{n-1}|$ in the third equality.
We hence get \eqref{anWill2} when $\Omega$ has no hole.

If $\Omega$ has holes, let $M^n\backslash\Omega$ be $N_\Omega$ unbounded
connected components $E_i$ ($i=1,2,3,\ldots, N_\Omega$). Set
$D=M^n\backslash(\cup^{N_\Omega}_{i=1}E_i)$. Then $D$ is a bounded open
set with smooth boundary, no holes and $\partial D\subseteq\partial\Omega$.
So for each $p\in \partial D$, we still have $H_f(p)\ge 0$ and $\partial_rf\ge-a$
along all minimal geodesic segments from $p$. Hence, by the preceding discussion,
we have
\[
\int_{\partial D}\left(\frac{H_f}{n-1}\right)^{n-1}e^{-f}d\sigma
\ge\mathrm{RV}_f(D) |\mathbb{S}^{n-1}|.
\]
By the definition of weighted relative volume, we have
$\mathrm{RV}_f(D)\ge\mathrm{RV}_f(\Omega)$ due to the fact $D\supseteq \Omega$
and $\partial D\subseteq\partial\Omega$. We also see that
\[
\int_{\partial \Omega}\left(\frac{H_f}{n-1}\right)^{n-1}e^{-f}d\sigma
\ge\int_{\partial D}\left(\frac{H_f}{n-1}\right)^{n-1}e^{-f}d\sigma.
\]
Putting these results together immediately gives \eqref{anWill2} for a general case.

Next we discuss the rigidity part of Theorem \ref{Ricf2}(a). Suppose
\begin{equation}\label{SWill2}
\mathrm{RV}_f(\Omega)=\frac{1}{|\mathbb{S}^{n-1}|
}\int_{\Sigma}\left(\frac{H_f}{n-1}\right)^{n-1}e^{-f}d\sigma>0.
\end{equation}
Since we assume $H_f=c$ is constant on $\Sigma$, then $H_f=c>0$ and we claim
that $\tau\equiv\infty$ on $\Sigma$. To prove this claim, we assume by
contradiction that there exists a point $p\in\Sigma$ such that $\tau(p)<\infty$.
Then from \eqref{weighetvol1}, we see that
\[
\mathrm{Vol}_f\{x\in M^n|d(x,\Omega)<R\}-|\Omega|_f\le\int_{\Sigma}\int_0^{\tau(p)} e^{ar-f(p)}\left(1+\frac{c r}{n-1}\right)^{n-1}drd\sigma(p)
\]
for $R>\tau(p)$. Dividing both sides by $m(R):=n |\mathbb{B}^n|\int^R_0e^{at}t^{n-1}dt$
and letting $R\rightarrow\infty$ yields
\[
\mathrm{RV}_f(\Omega)\le\lim_{R\to\infty}\frac{\int_{\Sigma}\int_0^{\tau(p)} e^{ar-f(p)}
\left(1+\frac{c r}{n-1}\right)^{n-1}drd\sigma(p)}{m(R)}=0,
\]
where we used $\tau(p)<\infty$, which contradicts \eqref{SWill2}.
Hence the claim follows.

Moreover, from the preceding discussion, we see that
\[
\int_{\Sigma}\left(\frac{H_f}{n-1}\right)^{n-1}e^{-f}d\sigma
\ge\int_{\partial D}\left(\frac{H_f}{n-1}\right)^{n-1}e^{-f}d\sigma
\ge\mathrm{RV}_f(D)|\mathbb{S}^{n-1}|\ge\mathrm{RV}_f(\Omega)|\mathbb{S}^{n-1}|.
\]
Combining this with \eqref{SWill2} implies that $\Sigma=\partial D$
and $\Omega$ has no hole. For any $R^{\prime}<R$, we apply the monotonicity
of $\theta_f(r,p)$ to estimate that
\begin{align*}
\mathrm{Vol}_f\{x\in M^n|d(x,\Omega)<R\}-|\Omega|_f
=&\int_{\Sigma}\int_0^R\mathcal{A}_f(r,p)drd\sigma(p)\\
=&\int_{\Sigma}\int_0^R\theta_f(r,p)e^{ar}\left(1+\frac{H_f(p)}{n-1}r\right)^{n-1}drd\sigma(p)\\
=&\int_{\Sigma}\int_{R^{\prime}}^R\theta_f(r,p)e^{ar}
\left(1+\frac{H_f(p)}{n-1}r\right)^{n-1}drd\sigma(p)\\
&+\int_{\Sigma}\int_0^{R^{\prime}}\theta_f(r,p)e^{ar}
\left(1+\frac{H_f(p)}{n-1}r\right)^{n-1}drd\sigma(p)\\
\le&\int_{\Sigma}\theta_f(R^{\prime},p)\int_{R^{\prime}}^Re^{ar}
\left(1+\frac{H_f(p)}{n-1}r\right)^{n-1}drd\sigma(p)\\
&+\int_{\Sigma}\int_0^{R^{\prime}}\theta_f(r,p)e^{ar}\left(1+\frac{H_f(p)}{n-1}r\right)^{n-1}drd\sigma(p).
\end{align*}
Dividing both sides by $m(R):=n |\mathbb{B}^n|\int^R_0e^{at}t^{n-1}dt$
and letting $R\rightarrow\infty$ yields
\[
\mathrm{RV}_f(\Omega)\le\frac{1}{|\mathbb{S}^{n-1}|}
\int_{\Sigma}\left(\frac{H_f(p)}{n-1}\right)^{n-1}\theta_f(R^{\prime},p)d\sigma(p),
\]
where we used the L'Hopital rule. Then letting $R^{\prime}\to\infty$,
we have
\[
\mathrm{RV}_f(\Omega)\le\frac{1}{|\mathbb{S}^{n-1}|}
\int_{\Sigma}\left(\frac{H_f(p)}{n-1}\right)^{n-1}\theta_f(\infty,p)d\sigma(p),
\]
where $\theta_f(\infty, p)=\lim_{r\rightarrow\infty}\theta_f(r,p)\le e^{-f(p)}$.
Moreover, we have assumed \eqref{SWill2} and hence we must have
$\theta_f(\infty,p)=e^{-f(p)}$ for almost every $p\in\Sigma$. It follows that
\begin{equation}\label{eqcase}
\mathcal{A}_f(r,p)=e^{ar-f(p)}\left(1+\frac{H_f(p)}{n-1}r\right)^{n-1}\quad \text{on}\quad [0,\infty)
\end{equation}
for almost every $p\in\Sigma$. By continuity this identity holds
for all $p\in\Sigma$.

By \eqref{eqcase}, inspecting the comparison argument of Lemma \ref{lem}(a)
and \eqref{compdic}, we get $\partial_rf\equiv 0$ along all minimal geodesic
segments from each point of $\Sigma$. Moreover, on $\Phi(\lbrack0,\infty)\times\Sigma)$,
\[
\mathrm{Hess}\, r=\frac{\Delta r}{n-1}g\quad\text{and}\quad
\Delta_f r=\frac{(n-1)H_f}{n-1+H_fr}.
\]
Since $\partial_rf\equiv 0$, then $\langle\nabla f,\nabla r\rangle=0$
and $H_f=H$. So the above equalities imply
\begin{equation}\label{nar}
\mathrm{Hess}\, r=\frac{\Delta r}{n-1}g=\frac{H}{n-1+Hr}g.
\end{equation}
Hence the second fundamental form $h$ satisfies
\[
h=\frac{H}{n-1}g_{\Sigma},
\]
where $H=c$. Now we claim that $\Sigma$ must be connected. In fact since $\overline{\Omega}$
is a connected compact manifold with $\mathrm{Ric}_f\ge0$ and $H_f\ge0$,
by Proposition \ref{num}, we conclude that either $\Sigma$ is connected or $\Sigma$
has two components. If $\Sigma$ has two components, from Proposition \ref{num} we
see $H\equiv 0$ on each component. Moreover, we have proven $\partial_rf\equiv 0$.
Thus, $H_f\equiv 0$ on $\Sigma$. This is impossible due to our assumption
\eqref{SWill2}. Therefore, $\Sigma$ is connected.

Since $\Phi$ is a diffeomorphism starting from $[0,\infty)\times\Sigma$
onto its image, we have the following form of the pullback metric
$\Phi^{\ast}g: dr^2+\eta_r$, where $\eta_r$ is a $r$-dependent family of
metrics on $\Sigma$ and $\eta_0=g_{\Sigma}$.  In terms of local coordinates
$\{x_1,\cdots,x_{n-1}\}$ on $\Sigma$, \eqref{nar} implies
\[
\frac{\partial}{\partial r}\eta_{ij}=\frac{2H}{n-1+Hr}\eta_{ij}.
\]
This further implies that
\[
\eta_r=\left(1+\frac{H}{n-1}r\right)^2g_{\Sigma}.
\]
Therefore we show that $\Phi(\lbrack0,\infty)\times\Sigma)$ is isometric to
\[
\left([r_0,\infty)\times\Sigma,dr^2+(\tfrac{r}{r_0})^2g_{\Sigma}\right),
\]
where $r_0=\frac{n-1}{H_f}$.

Finally we examine the sufficient part of Theorem \ref{Ricf2}. That is, under
the assumptions of Theorem \ref{Ricf2}, if $(M^n\backslash\Omega,g, e^{-f}dv)$
is isometric to \eqref{anwarprod} and $\partial_rf\equiv 0$ along all minimal
geodesic segments from each point of the connected $\Sigma=\partial\Omega$,
we inspect the equality of \eqref{anWill2}. On the one hand, using
$r_0=\frac{n-1}{H_f}$, we have
\[
\int_{\partial\Omega}\left(\frac{H_f}{n-1}\right)^{n-1}e^{-f}d\sigma
=\frac{1}{{r_0}^{n-1}}|\Sigma|_f.
\]
On the other hand, since $a=0$, we have
\begin{align*}
\mathrm{RV}_f(\Omega) |\mathbb{S}^{n-1}|&=\lim_{R\to\infty}
\frac{\int_{\Sigma}\int^{r_0+R}_{r_0}(\frac{r}{r_0})^{n-1}e^{-f(p)}drd\sigma(p)}{\int^R_0r^{n-1}dr}\\
&=\lim_{R\to\infty}
\frac{\int_{\Sigma}(\frac{r_0+R}{r_0})^{n-1}e^{-f(p)}d\sigma(p)}{R^{n-1}}\\
&=\frac{1}{{r_0}^{n-1}}|\Sigma|_f.
\end{align*}
Hence the equality of \eqref{anWill2} holds.
\end{proof}

In the rest of this section we prove the case (b) of Theorem \ref{Ricf2}.
\begin{proof}[Proof of Theorem \ref{Ricf2}(b)]
The proof is nearly the same as the case (a) of Theorem \ref{Ricf2}.
So we only sketch main steps. Similar to the arguments of Theorem
\ref{Ricf2}(a), we assume, without loss of generality, that $\Omega$ has
no hole. Lemma \ref{lem}(b) shows that
\[
\overline{\theta}_f(r,p):=\frac{\mathcal{A}_f(r,p)}{\left(1+\frac{H_f(p)}{n-1}r\right)^{n-1+4k}}
\]
is non-increasing in $r$ on $[0,\tau(p))$, which implies that
\[
\mathcal{A}_f(r,p)\le e^{-f(p)}\left(1+\frac{H_f(p)}{n-1}r\right)^{n-1+4k}
\]
for $r<\tau(p)$. Hence, for any $R>0$, we have
\begin{align*}
\mathrm{Vol}_f\{x\in M^n|d(x,\Omega)<R\}-|\Omega|_f=&
\int_{\Sigma}\int_0^{\min\{R,\tau(p)\}}\mathcal{A}_f(r,p)drd\sigma(p)\\
\le&\int_{\Sigma}\int_0^{\min\{R,\tau(p)\}}e^{-f(p)}
\left(1+\frac{H_f(p)}{n-1}r\right)^{n-1+4k}drd\sigma(p)\\
\le&\int_{\Sigma}\int_0^R e^{-f(p)}
\left(1+\frac{H_f(p)}{n-1}r\right)^{n-1+4k}drd\sigma(p).
\end{align*}
Dividing both sides by $|\mathbb{B}^{n+4k}|R^{n+4k}$ and letting
$R\rightarrow\infty$ yields
\begin{align*}
\overline{\mathrm{RV}}_f(\Omega)&\le\lim_{R\to\infty}\frac{\int_{\Sigma}\int_0^R e^{-f(p)}
\left(1+\frac{H_f(p)}{n-1}r\right)^{n-1+4k}drd\sigma(p)}{|\mathbb{B}^{n+4k}|r^{n+4k}}\\
&=\lim_{R\to\infty}\frac{\int_{\Sigma}e^{-f(p)}
\left(1+\frac{H_f(p)}{n-1}R\right)^{n-1+4k}d\sigma(p)}{(n+4k)|\mathbb{B}^{n+4k}|R^{n-1+4k}}\\
&=\frac{1}{|\mathbb{S}^{n-1+4k}|
}\int_{\Sigma}\left(\frac{H_f}{n-1}\right)^{n-1+4k}e^{-f}d\sigma,
\end{align*}
where we used the L'Hopital rule in the second equality and
$(n+4k)|\mathbb{B}^{n+4k}|=|\mathbb{S}^{n-1+4k}|$ in the third equality.
This completes the proof of \eqref{anWill2b}.

The rigidity discussion is also similar to the proof of the rigidity part of
Theorem \ref{Ricf2}(a). We easily see that the equality of
\eqref{anWill2b} implies
\begin{equation}\label{eqcaseb}
\mathcal{A}_f(r,p)=e^{-f(p)}\left(1+\frac{H_f(p)}{n-1}r\right)^{n-1+4k}\quad \text{on}\quad [0,\infty)
\end{equation}
for all $p\in\Sigma$. Using this equality and inspecting the comparison argument
of Lemma \ref{lem}(b), we get that $k=0$ and $f\equiv 0$ along all minimal
geodesic segments from $\Sigma$. Thus, we get $H_f=H=c>0$. Moreover, on $\Phi(\lbrack0,\infty)\times\Sigma)$,
\[
\mathrm{Hess}\,r=\frac{\Delta r}{n-1}g\quad\text{and}\quad
\Delta_f r=\frac{(n-1)H_f}{n-1+H_fr}.
\]
The rest of proof is the same as the case of Theorem \ref{Ricf2}(a).
We omit the details.
\end{proof}

%55555555555555555555555555555555555555555555555555555555555555555555555555555555555555555555555555555555555555555555555555555555555555555555555555555555555555

\section{Willmore-type inequality for $\mathrm{Ric}^m_f\ge 0$}\label{Rmfr}

In this section, we prove Theorem \ref{Ricmf} in introduction by adapting the
argument of Wang \cite{Wa}. We now assume $\mathrm{Ric}^m_f\ge 0$ and $f$ is
any smooth function on $M^n$.

\begin{proof}[Proof of Theorem \ref{Ricmf}]
The proof is similar to the argument of Theorem \ref{Ricf2}. We first
prove \eqref{Will} when $\Omega$ has no hole. By Lemma \ref{lemv2}, since $\Delta_fr=\frac{\mathcal{A}_f^{\prime}}{\mathcal{A}_f}$,
we conclude that
\[
\theta_f(r,p):=\frac{\mathcal{A}_f(r,p)}{\left(
1+\frac{H_f(p)}{m-1}r\right)^{m-1}}
\]
is non-increasing in $r$ on $[0,\tau(p))$, where $\tau(p)=\frac{m-1}{H_f^{-}(p)}$.
Since $\theta_f(0,p)=e^{-f(p)}$, we get $\theta_f(r,p)\le e^{-f(p)}$, that is
\[
\mathcal{A}_f(r,p)\le e^{-f(p)}\left(1+\frac{H_f(p)}{m-1}r\right)^{m-1}
\]
for $r\in[0,\tau(p))$. For any $R>0$, we apply this estimate to get that
\begin{equation}
\begin{aligned}\label{weighetvolm1}
\mathrm{Vol}_f\{x\in M^n|d(x,\Omega)<R\}-|\Omega|_f=&
\int_{\Sigma}\int_0^{\min(R,\tau(p))}\mathcal{A}_f(r,p)drd\sigma(p)\\
\le&\int_{\Sigma}\int_0^{\min(R,\tau(p))}
e^{-f(p)}\left(1+\frac{H_f(p)}{m-1}r\right)^{m-1}drd\sigma(p)\\
\le&\int_{\Sigma}\int_0^{\min(R,\tau(p))}e^{-f(p)}
\left(1+\frac{H_f^+(p)}{m-1}r\right)^{m-1}drd\sigma(p)\\
\le&\int_{\Sigma}\int_0^R e^{-f(p)}
\left(1+\frac{H_f^+(p)}{m-1}r\right)^{m-1}drd\sigma(p)\\
=&\frac{R^m}{m}\int_{\Sigma}e^{-f(p)}
\left(\frac{H_f^+(p)}{m-1}\right)^{m-1}d\sigma(p)+O(R^{m-1}),
\end{aligned}
\end{equation}
where $H_f^{+}(p):=\max\{H_f(p),\,0\}$. Dividing both sides by
$|\mathbb{B}^m|R^m=|\mathbb{S}^{m-1}|R^m/m$ and letting
$R\rightarrow\infty$, we get
\[
\mathrm{AVR}^m_f(g)\le\frac{1}{|\mathbb{S}^{m-1}|}\int_{\Sigma}\left(\frac{H_f^+}{m-1}\right)
^{m-1}e^{-f}d\sigma.
\]
Finally, using the fact $|H_f|=H^+_f+H^-_f$ we get \eqref{Will}.

We now show that if $\Omega$ has some holes, \eqref{Will} still holds.
Indeed if $\Omega$ has some holes, let $M^n\backslash\Omega$ be all
unbounded connected components $E_i$ ($i=1,2,3,\ldots$). Then
$D:=M^n\backslash(\cup_{i=1}E_i)$ is a bounded open set with smooth
boundary and no holes, and $\partial D\subseteq\partial\Omega$.
So for each $p\in \partial D$, by the preceding discussion, we have
\[
\int_{\partial D}\left|\frac{H_f}{m-1}\right|^{m-1}e^{-f}d\sigma
\ge\mathrm{AVR}^m_f(g) |\mathbb{S}^{m-1}|.
\]
Combining this with the fact $\partial D\subseteq \partial\Omega$ again, we
get \eqref{Will} for a general case.

Below we discuss the rigidity part of Theorem \ref{Ricmf}. Suppose $H_f=c$
is constant on $\Sigma$ and
\begin{equation}\label{SWill}
\mathrm{AVR}^m_f(g)=\frac{1}{|\mathbb{S}^{m-1}|
}\int_{\Sigma}\left|\frac{H_f}{m-1}\right|^{m-1}e^{-f}d\sigma>0.
\end{equation}
Then $H_f=c>0$ and we claim that $\tau\equiv\infty$ on $\Sigma$.
To prove this claim, we assume by contradiction that there exists
a point $p\in\Sigma$ such that $\tau(p)<\infty$.
From \eqref{weighetvolm1}, we see that
\[
\mathrm{Vol}_f\{x\in M^n|d(x,\Omega)<R\}-|\Omega|_f
\le\int_{\Sigma}\int_0^{\tau(p)}e^{-f(p)}\left(1+\frac{c r}{m-1}\right)^{m-1}drd\sigma(p)
\]
for $R>\tau(p)$. Dividing both sides by $|\mathbb{B}^m|R^m$
and letting $R\rightarrow\infty$ yields
\[
\mathrm{AVR}^m_f(g)\le\lim_{R\to\infty}
\frac{\int_{\Sigma}\int_0^{\tau(p)}e^{-f(p)}\left(1+\frac{c r}{m-1}\right)^{m-1}drd\sigma(p)}{|\mathbb{B}^m|R^m}=0,
\]
where we used $\tau(p)<\infty$, which contradicts \eqref{SWill}.
Hence the claim follows.

Moreover, similar to the argument of Theorem
\ref{Ricf2}(a), $\Omega$ has no hole. For any $R^{\prime}<R$, using the
monotonicity of $\theta_f(r,p)$, we have
\begin{align*}
\mathrm{Vol}_f\{x\in M^n|d(x,\Omega)<R\}-|\Omega|_f=&\int_{\Sigma}\int_0^R
\mathcal{A}_f(r,p)drd\sigma(p)\\
\le&\int_{\Sigma}\int_0^R
\theta_f(r,p)\left(1+\frac{H_f(p)}{m-1}r\right)^{m-1}drd\sigma(p)\\
\le &\int_{\Sigma}\int_{R^{\prime}}^R\theta_f(r,p)
\left(1+\frac{H_f(p)}{m-1}r\right)^{m-1}drd\sigma(p)\\
&+\int_{\Sigma}\int_0^{R^{\prime}}\theta_f(r,p)
\left(1+\frac{H_f(p)}{m-1}r\right)^{m-1}drd\sigma(p)\\
\le&\int_{\Sigma}\theta_f(R^{\prime},p)
\int_{R^{\prime}}^R\left(1+\frac{H_f(p)}{m-1}r\right)^{m-1}drd\sigma(p)\\
&+\int_{\Sigma}\int_0^{R^{\prime}}\theta_f(r,p)\left(1+\frac{H_f(p)}{m-1}r\right)^{m-1}drd\sigma(p).
\end{align*}
Dividing both sides by $|\mathbb{B}^m|R^m=|\mathbb{S}^{m-1}|R^m/m$ and
letting $R\rightarrow\infty$, we obtain
\[
\mathrm{AVR}^m_f(g)\leq\frac{1}{|\mathbb{S}^{m-1}|}
\int_{\Sigma}\left(\frac{H_f(p)}{m-1}\right)^{m-1}\theta_f(R^{\prime},p)d\sigma(p).
\]
Letting $R^{\prime}\to\infty$,
\[
\mathrm{AVR}^m_f(g)\leq\frac{1}{|\mathbb{S}^{m-1}|}
\int_{\Sigma}\left(\frac{H_f(p)}{m-1}\right)^{m-1}\theta_f(\infty,p)d\sigma(p),
\]
where $\theta_f(\infty, p)=\lim_{r\rightarrow\infty}\theta_f(r,p)\le e^{-f(p)}$.
Since we assume \eqref{SWill}, then $\theta_f(\infty,p)=e^{-f(p)}$ for almost
every $p\in\Sigma$. Hence,
\[
\mathcal{A}_f(r,p)=e^{-f(p)}\left(1+\frac{H_f(p)}{m-1}r\right)^{m-1}
\]
on $[0,\infty)$ for almost every $p\in\Sigma$. By continuity, this identity is in
fact true for all $p\in\Sigma$. By this equality, analysing the comparison
argument of Lemma \ref{lemv2}, on $\Phi(\lbrack0,\infty)\times\Sigma)$,
we get that
\[
\Delta_f r=-\frac{m-1}{m-n}\langle\nabla f,\nabla r\rangle
\]
which is equivalent to
\[
\Delta r=-\frac{n-1}{m-n}\langle\nabla f,\nabla r\rangle.
\]
On $\Phi(\lbrack0,\infty)\times\Sigma)$, we also have
\[
\mathrm{Hess}\,r=\frac{\Delta r}{n-1}g\quad\text{and}\quad
\Delta_f r=\frac{(m-1)H_f}{m-1+H_fr}.
\]
Putting these together, we have
\begin{equation}\label{na2r}
\mathrm{Hess}\,r=\frac{H_f}{m-1+H_fr}g.
\end{equation}
Since $H_f=c>0$ on $\Sigma$, this implies
that the second fundamental form $h$ satisfies
\[
h=\frac{H_f}{m-1}g_{\Sigma}.
\]
Since $\Phi$ is a diffeomorphism starting from $[0,\infty)\times\Sigma$
onto its image, we have the following form of the pullback metric
$\Phi^{\ast}g: dr^2+\eta_r$, where $\eta_r$ is a $r$-dependent family of
metrics on $\Sigma$ and $\eta_0=g_{\Sigma}$.  In terms of local coordinates
$\{x_1,\cdots,x_{n-1}\}$ on $\Sigma$, \eqref{na2r} implies
\[
\frac{\partial}{\partial r}\eta_{ij}=\frac{2H_f}{m-1+H_fr}\eta_{ij}.
\]
This further implies that
\[
\eta_r=\left(1+\frac{H_f}{m-1}r\right)^2g_{\Sigma}.
\]
Therefore we show that $\Phi(\lbrack0,\infty)\times\Sigma)$ is isometric to
\[
\left([r_0,\infty)\times\partial\Omega,dr^2+(\tfrac{r}{r_0})^2g_{\partial\Omega}\right)
\]
with $r_0=\frac{m-1}{H_f}$. Since $M^n$ has only one end,
then $\Sigma$ must be connected.
\end{proof}

%55555555555555555555555555555555555555555555555555555555555555555555555555555555555555555555555555555555555555555555555555555555555555555555555555555555555555

\section{Willmore-like inequality for shrinkers}\label{remshri}
In this section, we apply a similar strategy of proving Theorems \ref{Ricf2}
and \ref{Ricmf} to prove Theorem \ref{shri}. Here the computations are much more
involved than before.

\begin{proof}[Proof of Theorem \ref{shri}]
We first give an upper bound of mean curvature by following arguments
of \cite{WW, MuWa}. Similar to Lemma \ref{lem}, for any a fixed point
$p\in\partial\Omega$, let $\gamma_p(t)=\exp_pt\nu(p)$ be the normal
geodesic with initial velocity $\nu(p)$. Using $\mathrm{Ric}_f=\tfrac 12\, g$
in \eqref{ricca}, we have
\[
\frac{\partial}{\partial r}(\Delta r)+\frac{(\Delta r)^2}{n-1}\le f''(r)-\frac 12
\]
for $r\in[0,\tau(p))$, where $f''(r):=\mathrm{Hess} f(\partial r,\partial r)=\frac{d^2}{dr^2}(f\circ \gamma)(r)$. It is inequivalent to
\[
\frac{\tfrac{\partial}{\partial r}[(n-1+H(p)r)^2\Delta r]}{(n-1+H(p)r)^2}
+\frac{1}{n-1}\left[\Delta r-\frac{(n-1)H(p)}{n-1+H(p)r}\right]^2
\le\frac{(n-1)H(p)^2}{(n-1+H(p)r)^2}+f''(r)-\frac 12
\]
for $r\in[0,\tau(p))$. Discarding the above second nonnegative term, we have
\[
\frac{\partial}{\partial r}\left[(n-1+H(p)r)^2\Delta r\right]
\le(n-1)H(p)^2+f''(r)(n-1+H(p)r)^2-\frac 12\left(n-1+H(p)r\right)^2
\]
for $r\in[0,\tau(p))$. Integrating the above inequality from $0$ to $r>0$ and
using the initial condition $\Delta r|_{r=0}=H(p)$, we get that
\begin{align*}
(n-1&+H(p)r)^2\Delta r-(n-1)^2H(p)\\
&\le(n-1)H(p)^2r+\int^r_0(n-1+H(p)s)^2df'(s)-\frac 12\int^r_0(n-1+H(p)s)^2ds\\
&=(n-1)H(p)^2r+(n-1+H(p)r)^2f'(r)-(n-1)^2f'(0)\\
&\quad-2H(p)\int^r_0f'(s)(n-1+H(p)s)ds
-\frac{1}{6H(p)}\Big[(n-1+H(p)r)^3-(n-1)^3\Big]
\end{align*}
for $r\in[0,\tau(p))$. Rearranging some terms of the above inequality, we arrive at
\begin{equation}
\begin{aligned}\label{intesh}
\Delta r&\le\frac{(n-1)H(p)}{n-1+H(p)r}+f'(r)-\frac{n-1+H(p)r}{6H(p)}
+\frac{(n-1)^2}{[n-1+H(p)r]^2}\left[\tfrac{n-1}{6H(p)}-f'(0)\right]\\
&\quad-\frac{2H(p)}{[n-1+H(p)r]^2}\int^r_0f'(s)(n-1+H(p)s)ds
\end{aligned}
\end{equation}
for $r\in[0,\tau(p))$. Integrating the above inequality from $0$ to $r>0$
once again, and using
\[
\Delta r=(\ln\mathcal{A}(r))',
\]
where $\mathcal{A}(r)=\mathcal{A}(p,r)$, we have
\begin{align*}
\ln\mathcal{A}(r)&\le(n-1)\ln\left(1+\tfrac{H(p)}{n-1}r\right)
+[f(r)-f(0)]-\frac{1}{6H(p)}\left[(n-1)r-\tfrac{H(p)}{2}r^2\right]\\
&\quad+\frac{(n-1)^2}{H(p)[n-1+H(p)t]}\left[f'(0)-\tfrac{n-1}{6H(p)}\right]^{t=r}_{t=0}\\
&\quad-\int^r_0\frac{2H(p)}{(n-1+H(p)t)^2}\left(\int^t_0f'(s)[n-1+H(p)s]ds\right)dt
\end{align*}
for $r\in[0,\tau(p))$, where we used the fact $\mathcal{A}(0)=1$. Using
\begin{align*}
-\int^r_0&\frac{2H(p)}{(n-1+H(p)t)^2}\left(\int^t_0f'(s)[n-1+H(p)s]ds\right)dt\\
&=\frac{2}{n-1+H(p)t}\left(\int^t_0f'(s)[n-1+H(p)s]ds\right){\Big|}^{t=r}_{t=0}-2\int^r_0f'(s)ds,
\end{align*}
the above inequality can be written as
\begin{align*}
\ln\mathcal{A}(r)&\le(n-1)\ln\left(1+\tfrac{H(p)}{n-1}r\right)
-[f(r)-f(0)]-\frac{1}{6H(p)}\left[(n-1)r-\tfrac{H(p)}{2}r^2\right]\\
&\quad+\frac{(n-1)^2}{H(p)[n-1+H(p)r]}\left[f'(0)-\tfrac{n-1}{6H(p)}\right]
-\frac{n-1}{H(p)}\left[f'(0)-\tfrac{n-1}{6H(p)}\right]\\
&\quad+\frac{2}{n-1+H(p)r}\int^r_0f'(s)\Big[n-1+H(p)s\Big]ds
\end{align*}
for $r\in[0,\tau(p))$. Since $H(p)>0$, the above inequality furthermore implies that
\begin{align*}
-&\frac{2H(p)}{[n-1+H(p)r]^2}\int^r_0f'(s)\Big[n-1+H(p)s\Big]ds\\
&\le-\frac{H(p)\ln\mathcal{A}(r)}{n-1+H(p)r}+\frac{(n-1)H(p)}{n-1+H(p)r}\ln\left(1+\tfrac{H(p)}{n-1}r\right)
-\frac{H(p)[f(r)-f(0)]}{n-1+H(p)r}\\
&\quad-\frac{1}{6[n-1+H(p)r]}\left[(n-1)r-\tfrac{H(p)}{2}r^2\right]
+\frac{(n-1)^2}{[n-1+H(p)r]^2}\left[f'(0)-\tfrac{n-1}{6H(p)}\right]\\
&\quad-\frac{n-1}{n-1+H(p)r}\left[f'(0)-\tfrac{n-1}{6H(p)}\right]\\
&=-\frac{H(p)}{n-1+H(p)r}\ln\tfrac{\mathcal{A}(r)}{\big(1+\tfrac{H(p)}{n-1}r\big)^{n-1}}
-\frac{H(p)f(r)}{n-1+H(p)r}\\
&\quad-\frac{1}{6[n-1+H(p)r]}\left[(n-1)r-\tfrac{H(p)}{2}r^2\right]
+\frac{(n-1)^2}{[n-1+H(p)r]^2}\left[f'(0)-\tfrac{n-1}{6H(p)}\right]\\
&\quad+\frac{1}{n-1+H(p)r}\left[H(p)f(0)-(n-1)f'(0)+\tfrac{(n-1)^2}{6H(p)}\right]
\end{align*}
for $r\in[0,\tau(p))$. Substituting this into \eqref{intesh} we have
\begin{equation}
\begin{aligned}\label{intesh2}
\Delta r&\le\frac{(n-1)H(p)}{n-1+H(p)r}
-\frac{H(p)}{n-1+H(p)r}\ln\tfrac{\mathcal{A}(r)}{\big(1+\tfrac{H(p)}{n-1}r\big)^{n-1}}
+f'(r)-\frac{H(p)f(r)}{n-1+H(p)r}\\
&\quad-\frac{n-1+H(p)r}{6H(p)}
-\frac{1}{6[n-1+H(p)r]}\left[(n-1)r-\tfrac{H(p)}{2}r^2\right]\\
&\quad+\frac{1}{n-1+H(p)r}\left[H(p)f(0)-(n-1)f'(0)+\tfrac{(n-1)^2}{6H(p)}\right]
\end{aligned}
\end{equation}
for $r\in[0,\tau(p))$.

Next, we apply the above estimate to derive an upper bound of $\mathcal{A}(r)$.
On shrinkers, we have $\mathrm{R}+|\nabla f|^2=f$ and $\mathrm{R}\ge0$ (see
\cite{Cbl}), then $|\nabla f|^2\le f$. Using this, we observe that
\begin{equation}\label{fmo}
f'(r)\le\frac{H(p)|\nabla f|^2}{n-1+H(p)r}
+\frac{n-1+H(p)r}{4H(p)}\le\frac{H(p)f(r)}{n-1+H(p)r}
+\frac{n-1+H(p)r}{4H(p)}.
\end{equation}
where we used the Cauchy-Schwarz inequality. Putting this into \eqref{intesh2} yields
\begin{equation}
\begin{aligned}\label{intesh3}
\Delta r&\le\frac{(n-1)H(p)}{n-1+H(p)r}
-\frac{H(p)}{n-1+H(p)r}\ln\tfrac{\mathcal{A}(r)}{\big(1+\tfrac{H(p)}{n-1}r\big)^{n-1}}\\
&\quad+\frac{n-1+H(p)r}{12H(p)}
-\frac{1}{6[n-1+H(p)r]}\left[(n-1)r-\tfrac{H(p)}{2}r^2\right]\\
&\quad+\frac{1}{n-1+H(p)r}\left[H(p)f(0)-(n-1)f'(0)+\tfrac{(n-1)^2}{6H(p)}\right]\\
&=\frac{(n-1)H(p)}{n-1+H(p)r}
-\frac{H(p)}{n-1+H(p)r}\ln\tfrac{\mathcal{A}(r)}{\big(1+\tfrac{H(p)}{n-1}r\big)^{n-1}}\\
&\quad+\frac{1}{n-1+H(p)r}\left[H(p)f(0)-(n-1)f'(0)+\tfrac{(n-1)^2}{4H(p)}\right]
\end{aligned}
\end{equation}
for $r\in[0,\tau(p))$. If we let
\[
\theta(r,p):=\frac{\mathcal{A}(r)}{\big(1+\tfrac{H(p)}{n-1}r\big)^{n-1}},
\]
then \eqref{intesh3} can be written as
\begin{equation}\label{monton}
\left\{\Big[n-1+H(p)r\Big]\ln\theta(r,p)\right\}'\le c
\end{equation}
for $r\in[0,\tau(p))$, where $c:=H(p)f(p)-(n-1)f'(p)+\tfrac{(n-1)^2}{4H(p)}$.
We remark that $c\ge0$. Indeed, since $|\nabla f|^2\le f$, then
\[
c\ge H(p)f(p)-(n-1)\sqrt{f(p)}+\frac{(n-1)^2}{4H(p)}
=\left(\sqrt{H(p)f(p)}-\tfrac{n-1}{2\sqrt{H(p)}}\right)^2\ge0.
\]
Integrating \eqref{monton} from $0$ to $r>0$ gives
\begin{equation}\label{voeup}
\mathcal{A}(r,p)\le e^{\frac{c r}{n-1+H(p)r}}\left(1+\frac{H(p)}{n-1}r\right)^{n-1}
\end{equation}
for $r\in[0,\tau(p))$.

Finally, we apply \eqref{voeup} to prove our result. Without loss of generality,
we assume that $\Omega$ has no hole. For any $R>0$, we see that
\begin{equation}
\begin{aligned}\label{shrinkervol}
\mathrm{Vol}\{x\in M^n|d(x,\Omega)<R\}-|\Omega|=&
\int_{\Sigma}\int_0^{\min(R,\tau(p))}\mathcal{A}(r,p)drd\sigma(p)\\
\le&\int_{\Sigma}\int_0^{\min(R,\tau(p))}e^{\frac{c r}{n-1+H(p)r}}\left(1+\frac{H(p)}{n-1}r\right)^{n-1}drd\sigma(p)\\
\le&\int_{\Sigma}\int_0^Re^{\frac{c r}{n-1+H(p)r}}\left(1+\frac{H(p)}{n-1}r\right)^{n-1}drd\sigma(p).
\end{aligned}
\end{equation}
Dividing both sides by $|\mathbb{B}^n|R^n=|\mathbb{S}^{n-1}|R^n/n$,
letting $R\rightarrow\infty$, and using the L'Hopital rule,
we get
\begin{equation*}
\begin{aligned}
\mathrm{AVR}(g)|\mathbb{S}^{n-1}|\le&\frac{\int_{\Sigma}e^{\frac{c R}{n-1+H(p)R}}\left(1+\frac{H(p)}{n-1}R\right)^{n-1}drd\sigma(p)}{R^{n-1}}\\
\le&\int_{\Sigma}e^{\frac{c}{H(p)}}\left(\frac{H(p)}{n-1}\right)^{n-1}d\sigma(p),
\end{aligned}
\end{equation*}
where we used $c\ge0$. This completes the proof of the inequality.

Below we discuss the equality case. On the one hand, an open $n$-ball
$\Omega=\mathbb{B}^n(r_0)$ with radius $r_0>0$ in the Gaussian
shrinker $(\mathbb{R}^n, \delta_{ij},e^{-|x|^2/4}dv)$ attains the
equality of \eqref{shwill}. On the other hand, we will show that
such case is a unique example when the equality of \eqref{shwill} occurs.
Suppose
\begin{equation}\label{equshr}
\mathrm{AVR}(g)|\mathbb{S}^{n-1}|=\int_{\Sigma}e^{\frac{c}{H}}\left(\frac{H}{n-1}\right)^{n-1}d\sigma>0.
\end{equation}
By the same argument of Theorem \ref{Ricf2}(a), we know that $\Omega$ has no hole.
Moreover, we claim that $\tau\equiv\infty$ on $\Sigma$. To prove this claim,
we assume by contradiction that there exists a point $p\in\Sigma$ such that
$\tau(p)<\infty$. From \eqref{shrinkervol}, we get that
\[
\mathrm{Vol}\{x\in M^n|d(x,\Omega)<R\}-|\Omega|
\le\int_{\Sigma}\int_0^{\tau(p)}e^{\frac{c r}{n-1+H(p)r}}
\left(1+\frac{H(p)}{n-1}r\right)^{n-1}drd\sigma(p)
\]
for $R>\tau(p)$. Dividing both sides by $|\mathbb{B}^n|R^n$
and letting $R\rightarrow\infty$ yields
\[
\mathrm{AVR}(g)\le\lim_{R\to\infty}
\frac{\int_{\Sigma}\int_0^{\tau(p)}e^{\frac{c r}{n-1+H(p)r}}
\left(1+\frac{H(p)}{n-1}r\right)^{n-1}drd\sigma(p)}{|\mathbb{B}^n|R^n}=0,
\]
where we used $\tau(p)<\infty$, which contradicts \eqref{equshr}. Hence the claim follows.

For any a fixed point $p\in\Sigma$, let
\[
K(r,p):=(n-1+H(p)r)\ln \theta(r,p)-cr.
\]
By \eqref{monton}, $K(r,p)$ is non-increasing in $r$ on $[0,\infty)$.
Since $K(0,p)=0$, then $K(r,p)\le0$ on $[0,\infty)$ for each
$p\in\Sigma$. Below we \emph{assert} that
\[
K(r,p)\equiv0
\]
on $[0,\infty)$ for each $p\in\Sigma$. We now prove this assertion.
For any $R^{\prime}<R$, using the definition of $K(r,p)$, we have
\begin{align*}
\mathrm{Vol}\{x\in M^n|d(x,\Omega)<R\}-|\Omega|=&\int_{\Sigma}\int_0^R
\mathcal{A}(r,p)drd\sigma(p)\\
\le&\int_{\Sigma}\int_0^R
e^{\frac{K(r,p)+cr}{n-1+H(p)r}}\left(1+\frac{H(p)}{n-1}r\right)^{n-1}drd\sigma(p)\\
\le &\int_{\Sigma}\int_{R^{\prime}}^Re^{\frac{K(r,p)+cr}{n-1+H(p)r}}
\left(1+\frac{H(p)}{n-1}r\right)^{n-1}drd\sigma(p)\\
&+\int_{\Sigma}\int_0^{R^{\prime}}e^{\frac{K(r,p)+cr}{n-1+H(p)r}}
\left(1+\frac{H(p)}{n-1}r\right)^{n-1}drd\sigma(p)\\
=&\int_{\Sigma}e^{\frac{K(\bar{r},p)}{n-1+H(p)\bar{r}}}
\int_{R^{\prime}}^Re^{\frac{cr}{n-1+H(p)r}}\left(1+\frac{H(p)}{n-1}r\right)^{n-1}drd\sigma(p)\\
&+\int_{\Sigma}\int_0^{R^{\prime}}e^{\frac{K(r,p)+cr}{n-1+H(p)r}}\left(1+\frac{H(p)}{n-1}r\right)^{n-1}drd\sigma(p)
\end{align*}
for some $\bar{r}\in[R', R]$, where we used the mean value theorem in the last
equality. Dividing both sides by $|\mathbb{B}^n|R^n=|\mathbb{S}^{n-1}|R^n/n$ and
letting $R\rightarrow\infty$, we obtain
\[
\mathrm{AVR}(g)\leq\frac{1}{|\mathbb{S}^{n-1}|}
\int_{\Sigma}e^{\frac{c}{H(p)}}\left(\frac{H(p)}{n-1}\right)^{n-1}
e^{\frac{K(\bar{r},p)}{n-1+H(p)\bar{r}}}d\sigma(p),
\]
where we used $c\ge0$. Combining this with \eqref{equshr} yields
\[
\int_{\Sigma}e^{\frac{c}{H(p)}}\left(\frac{H(p)}{n-1}\right)^{n-1}
\left(e^{\frac{K(\bar{r},p)}{n-1+H(p)\bar{r}}}-1\right)d\sigma(p)\ge0.
\]
Combining this with the fact $K(r,p)\le0$ on $[0,\infty)$ for each $p\in\Sigma$, we
conclude that
\[
K(\bar{r},p)=0
\]
for each $p\in\Sigma$, where $\bar{r}\in[R', \infty)$. Since $R'$ can be
chosen to be arbitrarily large, the initial value $K(0,p)=0$ and the
non-increasing property of $K(r,p)$ in $r$, we must have $K(r,p)\equiv0$
on $[0,\infty)$ for any $p\in\Sigma$ and the assertion follows. This
assertion gives that
\[
\mathcal{A}(r,p)\equiv e^{\frac{c r}{n-1+H(p)r}}\left(1+\frac{H(p)}{n-1}r\right)^{n-1}
\]
for all $r\ge 0$ and each $p\in\Sigma$. By this equality, analysing the above comparison
argument, from \eqref{fmo}, we conclude that $|\nabla f|^2\equiv f$ on $(M^n,g,e^{-f}dv)$.
Thus the scalar curvature $\mathrm{R}\equiv 0$, and hence $(M^n,g,e^{-f}dv)$ is isometric
to the Gaussian shrinker $(\mathbb{R}^n, \delta_{ij},e^{-|x|^2/4}dv)$ due to \cite{PRS}.
Inspecting the above comparison argument, on $\Phi(\lbrack0,\infty)\times\Sigma)$, we
also have
\[
\mathrm{Hess}\,r=\frac{\Delta r}{n-1}g=\frac{H}{n-1+Hr}g\quad\text{and}\quad
\mathrm{Ric}\equiv 0.
\]
Following Wang's argument \cite{Wa}, the above first equation implies that
the second fundamental form $h$ satisfies
\[
h=\frac{H}{n-1}g_{\Sigma}.
\]
Let $\{e_0=\nu,e_1,...,e_{n-1}\}$ be orthnormal frame long $\Sigma$.
By the codazzi equation, we have
\[
R(e_k,e_j,e_i,\nu)=h_{ij,k}-h_{ik,j}=\frac{1}{n-1}(H_k\delta_{ij}-H_j\delta_{ik})
\]
for $1\le i,j,k\le n-1$. Tracing over $i$ and $k$ gives
\[
-\frac{n-2}{n-1}H_j=\mathrm{Ric}(e_j,\nu)=0.
\]
So $H$ is locally constant on $\Sigma$. Since $M^n=\mathbb{R}^n$ has only one end,
$\Sigma$ is connected. Hence $H$ is global constant on $\Sigma$. Thus,
$\Sigma$ is a closed embedded constant positive mean curvature hypersurface in
$\mathbb{R}^n$. By Alexandrov's theorem \cite{Al}, $\Sigma$ is a sphere.
\end{proof}

%55555555555555555555555555555555555555555555555555555555555555555555555555555555555555555555555555555555555555555555555555555555555555555555555555555555555555

\section{Isoperimetric type inequalities}\label{isoin}

In this section, we apply Theorems \ref{Ricf2} and \ref{Ricmf} to study
some weighted isoperimetric inequalities for a compact domain of a
weighted manifold provided it is a critical point of weighted isoperimetric
functional. First, we prove a sharp isoperimetric type inequality
in weighted manifolds with $\mathrm{Ric}_f\ge 0$.
\begin{theorem}\label{iso1}
Let $(M^n,g,e^{-f}dv)$ be a complete noncompact weighted $n$-manifold with
$\mathrm{Ric}_f\ge 0$, and let $\Omega\subset M^n$ be a bounded open set with
smooth boundary $\partial\Omega$. Let the weighted mean curvature $H_f$ of
$\partial\Omega$ be nonnegative everywhere. Assume that $\overline{\Omega}$
(the compact set of $\Omega$) is a critical point of the weighted isoperimetric
functional
\[
G\to\frac{|\partial G|_f^n}{|G|_f^{n-1}},
\]
where $G$ is a compact domain with smooth boundary $\partial G$.

(a) If \eqref{assump2} holds, then
\begin{equation}\label{isoineq1}
|\partial\Omega|_f\ge n|\mathbb{B}^n|^{\frac 1n}
\mathrm{RV}^{\frac 1n}_f(\Omega) |\Omega|_f^\frac{n-1}{n}.
\end{equation}
Moreover, if $\Omega$ is connected,, the equality of \eqref{isoineq1} holds if and only if $\partial\Omega$
is connected and $(M^n\backslash\Omega,g, e^{-f}dv)$ is isometric to
\[
\left([r_0,\infty)\times\partial\Omega,dr^2+(\tfrac{r}{r_0})^2g_{\partial\Omega}\right)
\]
with $\partial_rf\equiv 0$ along all minimal geodesic segments from
$\partial\Omega$ ($a=0$), where $r_0=\left(\frac{|\partial\Omega|_f}{\mathrm{RV}_f(\Omega)|\mathbb{S}^{n-1}|}\right)^{\frac{1}{n-1}}$.

(b) If \eqref{assump2b} holds, then
\begin{equation}\label{isoineq1b}
|\partial\Omega|_f\ge n\left(1+\tfrac{4k}{n}\right)^{\frac{1}{n+4k}}
\cdot|\mathbb{B}^{n+4k}|^{\frac{1}{n+4k}}
\cdot\overline{\mathrm{RV}}^{\frac{1}{n+4k}}_f(\Omega)
\cdot|\Omega|_f^\frac{n-1+4k}{n+4k}.
\end{equation}
Moreover, if $\Omega$ is connected, the equality of \eqref{isoineq1b} holds
if and only if $\partial\Omega$ is connected and $(M^n\backslash\Omega,g, e^{-f}dv)$
is isometric to
\[
\left([r_1,\infty)\times\partial\Omega,dr^2+(\tfrac{r}{r_1})^2g_{\partial\Omega}\right)
\]
with $ f\equiv 0$ along all minimal geodesic segments from
$\partial\Omega$ ($k=0$), where $r_1=\left(\frac{|\partial\Omega|_f}{\overline{\mathrm{RV}}_f(\Omega)|\mathbb{S}^{n-1}|}\right)^{\frac{1}{n-1}}$.
\end{theorem}

\begin{proof}[Proof of Theorem \ref{iso1}]
We only prove the case (a) because the case (b) can be similarly proved.
Assume that $\overline{\Omega}$ is a critical point of the weighted
isoperimetric functional
\[
G\to\frac{|\partial G|_f^n}{|G|_f^{n-1}}.
\]
Following the argument of Ros \cite{Ro}, let $\xi$ be any smooth function
on $\partial\Omega$ and consider the normal variation of $\partial\Omega$,
saying that $\psi(\cdot,t):\partial\Omega\to M^n$ defined by
\[
\psi(x,t)=\exp_x(-t\xi(x)\nu(x)), \quad x\in\partial\Omega,\,\, t>0,
\]
where $\exp$ is the exponential map of $M^n$ and $\nu:\partial\Omega\to TM^n$
is the outer unit and normal to $\partial\Omega$ smooth vector files along
the boundary. Clearly,  $\psi$ determines a variation $\overline{\Omega}_t$
of $\overline{\Omega}$, $|t|<\epsilon$. We set
\[
A_f(t)=|\partial\overline{\Omega}_t|_f
\quad \text{and}\quad
V_f(t)=|\overline{\Omega}_t|_f.
\]
Then by the first variational formulas, we have
\[
A_f'(0)=\int_{\partial\Omega}\xi H_f\,d\mu_{\partial\Omega} \quad \text{and} \quad
V_f'(0)=\int_{\partial\Omega}\xi\, d\mu_{\partial\Omega}.
\]
By our theorem assumption, we have
\[
\frac{d}{dt}\bigg|_{t=0}\frac{A_f(t)^n}{V_f(t)^{n-1}}=0,
\]
that is,
\[
\left(\frac{|\partial\Omega|_f}{|\Omega|_f}\right)^{n-1}
\int_{\partial\Omega}\xi\left[nH_f-(n-1)\frac{|\partial\Omega|_f}{|\Omega|_f}\right]d\mu_{\partial\Omega}=0
\]
for any given $\xi\in C^\infty(\partial\Omega)$. Hence,
we have
\begin{equation}\label{mecur1}
\frac{H_f}{n-1}=\frac{1}{n}\frac{|\partial\Omega|_f}{|\Omega|_f}.
\end{equation}
Substituting this into Theorem \ref{Ricf2}(a) yields
\begin{equation}\label{mecequ}
\left(\frac 1n\frac{|\partial\Omega|_f}{|\Omega|_f}\right)^{n-1}|\partial\Omega|_f
\ge\mathrm{RV}_f(\Omega)|\mathbb{S}^{n-1}|,
\end{equation}
which implies \eqref{isoineq1} by using $|\mathbb{S}^{n-1}|=n|\mathbb{B}^n|$.

The equality case of \eqref{isoineq1} and \eqref{mecur1} imply $\mathrm{RV}_f(\Omega)>0$,
$H_f>0$ is constant and the equality of \eqref{anWill2}. Since $\Omega$ is
connected, by Theorem \ref{Ricf2}(a), we get that $\partial\Omega$ is connected,
$\partial_rf\equiv 0$ and $(M^n\backslash\Omega,g)$ is isometric to the warped
product manifold given by \eqref{anwarprod}.
\end{proof}

In the end of this section, we give another sharp weighted isoperimetric type
inequality in weighted manifolds with $\mathrm{Ric}^m_f\ge 0$.
\begin{theorem}\label{iso2}
Let $(M^n,g,e^{-f}dv)$ be a complete noncompact weighted $n$-manifold with
$\mathrm{Ric}^m_f\ge 0$, and $\Omega\subset M^n$ a bounded open set with
smooth connected boundary $\partial\Omega$. If the compact set
$\overline{\Omega}$ is a critical point of the weighted isoperimetric
functional
\[
G\to\frac{|\partial G|_f^m}{|G|_f^{m-1}},
\]
where $G$ is a compact domain with smooth boundary $\partial G$,
then
\begin{equation}\label{isoineq2}
|\partial\Omega|_f\ge m|\mathbb{B}^m|^{\frac 1m}
(\mathrm{AVR}^m_f(g))^{\frac 1m} |\Omega|_f^\frac{m-1}{m}.
\end{equation}
Moreover, if $M^n$ has only one end, then the equality of \eqref{isoineq2}
implies that $\Omega$ is isometric to the Euclidean ball $\mathbb{B}^n(r_0)$
for some $r_0>0$ with $f$ being constant in $\Omega$ and
$(M^n\backslash\Omega,g, e^{-f}dv)$ is isometric to
\[
\left([r_0,\infty)\times\partial\Omega,dr^2+(\tfrac{r}{r_0})^2g_{\partial\Omega}\right).
\]
\end{theorem}

\begin{proof}[Proof of Theorem \ref{iso2}]
Assume that $\overline{\Omega}$ is a critical point of the weighted
isoperimetric functional
\[
G\to\frac{|\partial G|_f^m}{|G|_f^{m-1}}.
\]
Following the same argument of Theorem \ref{iso1}, by our theorem
assumption, we finally get
\[
\left(\frac{|\partial\Omega|_f}{|\Omega|_f}\right)^{m-1}
\int_{\partial\Omega}\xi\left[mH_f-(m-1)\frac{|\partial\Omega|_f}{|\Omega|_f}\right]d\mu_{\partial\Omega}=0
\]
for any given $\xi\in C^\infty(\partial\Omega)$. Hence,
\begin{equation}\label{mecur2}
\frac{H_f}{m-1}=\frac{1}{m}\frac{|\partial\Omega|_f}{|\Omega|_f}.
\end{equation}
Substituting this into Theorem \ref{Ricmf} immediately gives \eqref{isoineq2}
by using $|\mathbb{S}^{m-1}|=m|\mathbb{B}^m|$.

Below we analyze the equality case of \eqref{isoineq2}. On the one hand, since
$(\overline{\Omega},g,e^{-f}dv)$ has $\mathrm{Ric}^m_f\ge 0$ and \eqref{mecur2},
by Proposition \ref{inteinequ}, we know that $m=n$, $\mathrm{Ric}\ge 0$, $f$
is constant in $\overline{\Omega}$, and $\overline{\Omega}$ is isometric to
a closed Euclidean ball $\overline{\mathbb{B}}^n(r_0)$. On the other hand,
the equality of \eqref{isoineq2} and \eqref{mecur2} imply
$\mathrm{AVR}^m_f(g)>0$, $H_f$ is positive constant and the equality of
\eqref{Will}. We also assume $M^n$ has only one end. Putting these information
together, we apply Theorem \ref{Ricmf} to conclude that $(M^n\backslash\Omega,g)$
is isometric to the warped product manifold given by \eqref{warprod2}.
\end{proof}

In Theorems \ref{iso1} and \ref{iso2}, if $f$ is constant, our results
return to the manifold case of \cite{Ro}. Besides, the proof of Theorem \ref{iso2}
indicates that a compact hypersurface embedded in a weighted manifold admitting
a critical point of the weighted isoperimetric functional implies that $H_f$
is a positive constant (see \eqref{mecur2}). Combining this fact with
Proposition \ref{inteinequ}, we indeed get that
\begin{corollary}\label{corol}
Let $(M^n,g,e^{-f}dv)$ ($n\ge3$) be a complete noncompact weighted
$n$-manifold with $\mathrm{Ric}^m_f\ge 0$, and $\Omega\subset M^n$ a
bounded open set with smooth boundary $\partial\Omega$. If the compact
set $\overline{\Omega}$ is a critical point of the weighted isoperimetric
functional
\[
G\to\frac{|\partial G|_f^m}{|G|_f^{m-1}},
\]
where $G$ is a compact domain with smooth boundary $\partial G$,
then $\Omega$ is isometric to an Euclidean $n$-ball.
\end{corollary}

\

\textbf{Acknowledgements}.
The authors  sincerely thank the anonymous referee for pointing out a mistake
in Lemma \ref{lem}, many inaccurate statements and mini errors in preliminary
version of the paper, and making valuable suggestions which helped to improve
the presentation of this work. G.-Q. Wu's research is supported by Natural
Science Foundation of Zhejiang Province (No. LY23A010016), the Fundamental
Research Funds of Zhejiang Sci-Tech University (No. 24062095-Y), and the
Open Research Fund of Key Laboratory of Analytical Mathematics and Applications
(Fujian Normal University)(No. JAM2401). J.-Y. Wu's research is supported by Natural
Science Foundation of China (No. 12571144).

%55555555555555555555555555555555555555555555555555555555555555555555555555555555555555555555555555555555555555555555555555555555555555555555555555555555555555


\begin{thebibliography}{99}
\bibitem{AFM}V. Agostiniani, M. Fogagnolo, L. Mazzieri, Sharp geometric
inequalities for closed hypersurfaces in manifolds with nonnegative Ricci
curvature. Invent. Math. 222 (2020), no. 3, 1033-1101.

\bibitem{AM1}V. Agostiniani, L. Mazzieri, On the geometry of the level sets
of bounded static potentials, Comm. Math. Phys. 355 (2017), no.1, 261-301.

\bibitem{AM2}V. Agostiniani, L. Mazzieri, Monotonicity formulas in potential
theory, Calc. Var. Partial Differential Equations 59 (2020), Paper No. 6, 32 pp.

\bibitem{Al} A.D. Alexandrov,: A characteristic property of spheres,
Ann. Math. Pura Appl. 58 (1962), 303-315.

\bibitem{BE}D. Bakry, M. Emery, Diffusion hypercontractivitives,
in: S\'{e}minaire de Probabilit\'{e}s XIX, 1983/1984, in: Lecture
Notes in Math., vol. 1123, Springer-Verlag, Berlin, 1985, pp.
177-206.

\bibitem{BQ}D. Bakry, Z.-M. Qian, Volume comparison theorems without Jacobi fields,
in Current Trends in Potential Theory, 115-122. Theta Ser. Adv. Math., 4. Theta, Bucharest (2005).

\bibitem{BK}Z.M. Balogh, A. Krist\'aly, Sharp geometric inequalities in spaces with
nonnegative Ricci curvature and Euclidean volume growth, Math. Ann. 385 (2023), 1747-1773.

\bibitem{BCP} M. Batista, M.P. Cavalcante, J. Pyo, Some isoperimetric inequalities
and eigenvalue estimates in weighted manifolds, J. Math. Anal. Appl. 419 (2014), 617-626.

\bibitem{BF} S. Borghini, M. Fogagnolo, Comparison geometry for substatic manifolds
and a weighted isoperimetric inequality, arXiv:2307.14618v1.

\bibitem{Br}S. Brendle, Constant mean curvature surfaces in warped product manifolds,
 Publ. Math. Inst. Hautes \'Etudes Sci. 117 (2013), 247-269.

\bibitem{Br2}S. Brendle, Sobolev inequalities in manifolds with nonnegative curvature,
Comm. Pure. Appl. Math. 76 (2023), 2192-2218.


\bibitem{Cao} H.-D. Cao, Recent progress on Ricci solitons, Recent Advances in Geometric Analysis.
In: Y-I. Lee, C-S. Lin, M-P. Tsui, (eds.) Advanced Lectures in Mathematics (ALM), vol. 11, pp. 1-8.
International Press, Somerville, 2010.

\bibitem{CSW} J. Case, Y.-S. Shu, G.-F. Wei, Rigidity of quasi-Einstein metrics,
Diff. Geo. Appl. 29 (2011), 93-100.

\bibitem{CM} C. Cederbaum, A. Miehe, A new proof of the Willmore inequality via
a divergence inequality, (2024), arXiv: 2401.11939v1.

\bibitem{Cbl}B.-L. Chen, Strong uniqueness of the Ricci flow, J. Diff. Geom.
82 (2009), 363-382.

\bibitem{Chen} B.-Y. Chen, On a theorem of Fenchel-Borsuk-Willmore-Chern-Lashof.
Math. Ann. 194 (1971), 19-26.

\bibitem{Chen2} B.-Y. Chen, On the total curvature of immersed manifolds, I:
an inequality of Fenchel-Borsuk-Willmore, Amer. J. Math. 93 (1971), 148-162.

\bibitem{CLY}B. Chow, P. Lu, B. Yang,  A necessary and sufficient condition for
Ricci shrinkers to have positive AVR,  Proc. AMS. 140 (2012), 2179-2181.

\bibitem{FLZ} F.-Q. Fang, X.-D. Li, Z.-L. Zhang, Two generalizations of Cheeger-Gromoll splitting
theorem via Bakry-Emery Ricci curvature, Ann. Inst. Fourier (Grenoble) 59 (2009), 563-573.

\bibitem{FM} M. Fogagnolo, L. Mazzieri, Minimising hulls, $p$-capacity and isoperimetric
inequality on complete Riemannian manifolds, J. Funct. Anal. 283 (2022), no. 9, Paper No.
109638, 49 pp.

\bibitem{FW} X.-N. Fu, J.-Y. Wu, Harmonic functions and end numbers on smooth
metric measure spaces, J. Korean Math. Soc. 62 (2025), 1-31.

\bibitem{Ham} R. Hamilton, The formation of singularities in the Ricci flow,
Surveys in Differential Geometry, International Press, Boston, vol. 2, (1995), 7-136.

\bibitem{HW}F.-B. Hang, X.-D. Wang, Vanishing sectional curvature on the boundary and
a conjecture of Schroeder and Strake, Pacific J. Math. 232 (2007), 283-287.

\bibitem{HK} E. Heintze, H. Karcher, A general comparison theorem with applications
to volume estimates for submanifolds, Ann. Sci. Norm. Sup\'er. 11 (1978), 451-470.

\bibitem{Hu} Y.-X. Hu, Willmore inequality on hypersurfaces in hyperbolic space,
Proc. Amer. Math. Soc. 146 (2018), no.6, 2679-2688.

\bibitem{Ic} R. Ichida, Riemannian manifolds with compact boundary, Yokohama Math. J.
29 (1981), 169-177.

\bibitem{JY} X.-S. Jin, J.-B. Yin, Willmore-type inequality for closed hypersurfaces
in complete manifolds with Ricci curvature bounded below, (2024), arXiv:2402.02465v1.

\bibitem{Jo} F. Johne, Sobolev inequalities on manifolds with nonnegative
Bakry-\'Emery Ricci curvature, (2021), arXiv:2103.08496.

\bibitem{Kas}A. Kasue, Ricci curvature, geodesics and some geometric properties of
riemannian manifolds with boundary, J. Math. Soc. Jpn. 35 (1983), 117-131.

\bibitem{Kle} B. Kleiner, An isoperimetric comparison theorem, Invent. Math.
108 (1) (1992), 37-47.

\bibitem{KM} A. V. Kolesnikov and E. Milman, Brascamp-Lieb-type inequalities
on weighted Riemannian manifolds with boundary, J. Geom. Anal. 27 (2017), no. 2, 1680-1702.

\bibitem{LX} J.-F, Li, C. Xia, An integral formula and its applications on sub-static
manifolds, J. Diff. Geom. 113 (2019), 493–518.

\bibitem{Lo}J. Lott, Some geometric properties of the Bakry-\'Emery-Ricci
tensor, Comment. Math. Helv. 78 (2003), 865-883.

\bibitem{LV} J. Lott, C. Villani, Ricci curvature for metric-measure spaces via optimal transport,
Ann. Math. 169 (2009), 903-991.

\bibitem{MuWa} O. Munteanu, J.-P. Wang, Geometry of manifolds with densities, Adv. Math. 259 (2014), 269-305.

 \bibitem{PRS} S. Pigola, M. Rimoldi, A.G. Setti, Remarks on non-compact gradient Ricci solitons,
 Math. Z. 268 (2011), 777-790.

\bibitem{QX} G.-H, Qiu, C. Xia, A generalization of Reilly's formula and its applications to a new
Heintze-Karcher type inequality, Int. Math. Res. Not. (2015), 7608-7619.

\bibitem{Re} R.C. Reilly, Applications of the Hessian operator in a Riemannian manifold,
Indiana Univ. Math. J. 26 (1977), 459-472.

\bibitem{Ro} A. Ros, Compact hypersurfaces with constant higher order mean curvatures,
Rev. Mat. Iberoamericana 3 (1987), 447-453.

\bibitem{Ru} A. Rudnik, A sharp geometric inequality for closed hypersurfaces
in manifolds with asymptotically nonnegative curvature, (2023), arXiv:2310.08245.

\bibitem{Sc}F. Schulze, Optimal isoperimetric inequalities for surfaces in any
codimension in Cartan-Hadamard manifolds, Geom. Funct. Anal. 30 (2020), 255-288.

\bibitem{V} C. Viana, Isoperimetry and volume preserving stability in real
projective spaces, J. Diff. Geom. 125 (2023), 187-205.

\bibitem{WX} G.-F. Wang,  C. Xia, Uniqueness of stable capillary hypersurfaces
in a ball, Math. Ann. 374 (2019), 1845-1882.

\bibitem{Wa}X.-D. Wang, Remark on an inequality for closed hypersurfaces in
complete manifolds with nonnegative Ricci curvature, Ann. Fac. Sci. Toulouse
Math. (6) 32 (2023), no. 1, 173–178.

\bibitem{WW}G.-F. Wei, W. Wylie, Comparison geometry for the Bakry-Emery Ricci tensor,
J. Diff. Geom. 83 (2009), 377-405.

\bibitem{Wu} J.-Y. Wu, $L^p$-Liouville theorems on complete smooth metric measure spaces, Bull. Sci. Math. 138 (2014), 510-539.

\bibitem{Wu1} J.-Y. Wu, Elliptic gradient estimates for a weighted heat equation and applications, Math. Z. 280 (2015), 451-468.

\bibitem{Wu2} J.-Y. Wu, Counting ends on complete smooth metric measure spaces,
Proc. Amer. Math. Soc. 144 (2016), 2231-2239.

\bibitem{Wu3} J.-Y. Wu, Myers' type theorem with the Bakry-\'Emery Ricci tensor,
 Ann. Global Anal. Geom. 54 (2018), 541-549.

\bibitem{WWu1} J.-Y. Wu, P. Wu, Heat kernel on smooth metric measure spaces with nonnegative curvature, Math. Ann. 362 (2015),  717-742.

\bibitem{WWu2}J.-Y. Wu, P. Wu, Heat kernel on smooth metric measure spaces and applications, Math. Ann. 365 (2016), 309-344.

\bibitem{Wil} T.~J. Willmore, Mean curvature of immersed surfaces,
An. Sti. Univ. ``All. I. Cuza'' Iasi Sect. I a Mat.(N.S.), 14 (1968), 99-103.
\end{thebibliography}
\end{document}